\documentclass[11pt]{article}
\usepackage{IEEEtrantools}
\usepackage{nccmath}
\usepackage{mathtools}
\usepackage{physics}
\usepackage{latexsym}
\usepackage{amsfonts}
\usepackage{amssymb}
\usepackage{comment}
\usepackage{eqnarray, amscd, amsmath,amsfonts,amsthm}
\usepackage{mathrsfs}
\usepackage{enumerate}
\usepackage{dsfont}
\usepackage{hyperref}
\usepackage{graphicx}
\usepackage{epstopdf}
\usepackage{tikz}
\usepackage[T1]{fontenc}

\usetikzlibrary{shapes,arrows.meta,positioning}
\usepackage{cleveref}
\usepackage{tikz-cd}
\usepackage{psfrag}

\usepackage{multirow}
\usepackage{tikz}
\usetikzlibrary{arrows,decorations.markings}
\usepackage{bookmark}
\usepackage{hyperref}
\hypersetup{
     colorlinks   = true,
     citecolor    = blue
}

\newtheorem{dfn}{Definition}[section]
\newtheorem{thm}[dfn]{Theorem}
\newtheorem{lmma}[dfn]{Lemma}
\newtheorem{ppsn}[dfn]{Proposition}
\newtheorem{crlre}[dfn]{Corollary}
\newtheorem{xmpl}[dfn]{Example}
\newtheorem{rmrk}[dfn]{Remark}
\newtheorem{notation}[dfn]{Notation}

\newcommand{\la}{{\lambda}}
\newcommand{\bbc}{\mathbb{C}}
\newcommand{\bbz}{\mathbb{Z}}
\newcommand{\bbn}{\mathbb{N}}

\newcommand{\sk}{\mathcal{K}}
\newcommand{\cK}{\mathcal{K}}
\newcommand{\cB}{\mathcal{B}}
\newcommand{\cb}{\mathcal{B}}

\def \qed { \mbox{}\hfill
$\Box$\vspace{1ex}}

\providecommand{\Dref}[1]{Definition~\ref{#1}}
\providecommand{\Tref}[1]{Theorem~\ref{#1}}
\providecommand{\Lref}[1]{Lemma~\ref{#1}}
\providecommand{\Pref}[1]{Proposition~\ref{#1}}
\providecommand{\Yref}[1]{Corollary~\ref{#1}}
\providecommand{\Sref}[1]{Section~\ref{#1}}
\providecommand{\Eref}[1]{(\ref{#1})} 

\title{Sections and Chapters}
\let\oldthebibliography\thebibliography
\renewcommand{\thebibliography}[1]{
  \oldthebibliography{#1}
  \setlength{\itemsep}{0pt}
  \setlength{\parsep}{0pt}
  \setlength{\parskip}{0pt}
}

\begin{document}

\tikzset{->-/.style={decoration={
  markings,
  mark=at position #1 with {\arrow{>}}},postaction={decorate}}}
  \tikzset{-<-/.style={decoration={
  markings,
  mark=at position #1 with {\arrow{<}}},postaction={decorate}}}

\author{\sc{Vibhor Bhatt,\,\,Satyajit Guin,\,\,Palash Sharma}}
\title{Extensions of Operator Systems, Hyperrigidity and Quantum Double Suspension}

\date{}
\maketitle

\begingroup
\renewcommand{\thefootnote}{}
\begin{NoHyper}
\footnotetext{%

}
\end{NoHyper}
\endgroup

\begin{abstract}
We introduce a class of extensions of unital operator systems satisfying a certain Toeplitz type property and show that these extensions induce corresponding extensions of the $C^*$-algebras they generate. We prove that an extension in this class is hyperrigid in the $C^*$-algebra it generates if and only if both operator systems from which it is constructed are hyperrigid in the $C^*$-algebras they generate. We then introduce the quantum double suspension of operator systems as a special case of this construction and show that its iterates yield the Hong–Szyma\'nski quantum double suspensions at the $C^*$-algebra level, including the Vaksman–Soibelman quantum spheres. In particular, we construct hyperrigid operator systems that generate the $C^*$-algebras of these quantum spheres. As a further application, we show that every counterexample to Arveson’s hyperrigidity conjecture gives rise to a sequence of counterexamples. We identify the \v{S}ilov ideal and the $C^*$-envelope of these extensions and derive a $K$-theoretic condition under which each such sequence contains infinitely many pairwise non-isomorphic counterexamples. In particular, Scherer's counterexample gives rise to infinitely many pairwise non-isomorphic counterexamples. Finally, we compute the propagation number, an invariant of operator systems introduced by Connes and van Suijlekom, for this class of extensions.
\end{abstract}

\hypersetup{linkcolor=blue}

\section{Introduction}

The theory of extensions is a fundamental part of the structure theory of $C^*$-algebras. An extension of a $C^*$-algebra $A$ by a $C^*$-algebra $B$ is a short exact sequence
\begin{center}
$0\longrightarrow B\longrightarrow C\longrightarrow A\longrightarrow 0.$
\end{center}
The study of such extensions provides a framework for understanding $C^*$-algebras in terms of an ideal and the corresponding quotient. The theory has been developed extensively, with important connections to $K$-homology \cite{BDF,HR} and Kasparov's $KK$-theory \cite{Kas}. For an account of the general theory, see, for example, \cite{Bla}.

Among the many classes of $C^*$-algebra extensions, Toeplitz type extensions form a particularly concrete and useful class. The fundamental example is the Toeplitz algebra $\mathcal{T}$, which is an extension of $C(S^1)$ by $\sk(\ell^2(\bbn))$. More generally, one can consider extensions of the form
\[
0\longrightarrow \sk\otimes B\longrightarrow \mathscr{C}\longrightarrow A\longrightarrow 0,
\]
where $A$ and $B$ are unital $C^*$-algebras, subject to additional Toeplitz type conditions. Christensen and Ivan \cite{CI} investigated a special class of such extensions, corresponding to $B=\bbc$, from the geometric and metric perspective of Connes and Rieffel \cite{C1,C,Rie1,Rie2}. Subsequently, Hawkins and Zacharias \cite{HZ}, motivated in part by Kasparov's $KK$-theory and the work of Christensen and Ivan, formally introduced Toeplitz type $C^*$-extensions and characterized them in terms of a factorization of the Busby invariant through the Calkin algebra. They also showed that this class of extensions is sufficiently broad to encompass many examples of interest.

A noteworthy concrete class of Toeplitz type $C^*$-extensions is provided by the quantum double suspension (QDS) introduced by Hong and Szyma\'nski \cite{HS}. They showed that the $C^*$-algebras underlying the Vaksman--Soibelman \cite{VS} odd-dimensional quantum spheres $C(S_q^{2\ell+1}),\,0<q<1$, which are defined as quantum homogeneous spaces of Woronowicz's quantum group $SU_q(\ell+1)$ \cite{W1,W2}, can be obtained by iterating QDS, $\Sigma^2(C(S_q^{2\ell-1}))=C(S_q^{2\ell+1})$, starting with $B=C(S^1)$. In particular, $C(S_q^3)=\Sigma^2(C(S^1))=C(SU_q(2))$ is the $C^*$-algebra of continuous functions on the compact quantum group $SU_q(2)$ discovered by Woronowicz \cite{W}. Similarly, the $C^*$-algebras underlying the even-dimensional quantum spheres $C(S_q^{2\ell})$, which are quotients of the odd-dimensional quantum spheres, are obtained by iterating QDS, starting with $B=\bbc\oplus\bbc$, the $C^*$-algebra of a two-point space. In particular, $C(S_{q,c}^2)=\Sigma^2(\bbc\oplus\bbc)$ is the $C^*$-algebra of continuous functions on the generic Podle\'s spheres $S_{q,c}^2,\,c\in(0,\infty]$ \cite{Pod}. Beyond their significance in $C^*$-algebra theory, QDS also arise in noncommutative geometry \cite{ChSu,ChGu}, in connection with the local index formula of Connes and Moscovici \cite{CM}.

Motivated by the works of Connes and van Suijlekom \cite{CS,CS1}, in which noncommutative geometry and operator systems are connected in a fundamental way through \emph{spectral truncations} \cite{DLM,LS,BGP} and \emph{tolerance relations} \cite{DLL}, as well as by the recent investigation by the first two authors into the metric aspects of Toeplitz type $C^*$-extensions \cite{BGB}, in this article we introduce and investigate Toeplitz type extensions in the category of operator systems. Extensions of operator systems are considerably more subtle than extensions of $C^*$-algebras. Indeed, a surjective unital completely positive (UCP) map $\sigma:E\to F$ between operator systems $E\mbox{ and }F$, with kernel $J:=\ker(\sigma)$, does not automatically imply that $E/J$ is completely order isomorphic to $F$. One must establish that $\sigma$ is a complete quotient map \cite{KPTT1}.

More precisely, given two unital operator systems $E\subseteq\mathcal{B}(H_E)\mbox{ and }F\subseteq\mathcal{B}(H_F)$, and an infinite-dimensional orthogonal projection $P\in\mathcal{B}(H_E)$ satisfying certain Toeplitz type conditions, we construct a unital operator system $X(F;E,P)\subseteq\mathcal{B}(PH_E\otimes H_F)$, which we call a Toeplitz type extension, establishing the following:
\begin{enumerate}[$(i)$]
\item There is a surjective UCP $\sigma:X(F;E,P)\to E$ with $\ker(\sigma)$ of the form $\sk\otimes_{\min}F$ such that $E \cong_{\mbox{\tiny coi}} X(F;E,P) / \ker(\sigma)$, or in other words, $\sigma$ is a complete quotient map;
\item At the $C^*$-algebra level, $C^*(X(F;E,P))$ is a Toeplitz type $C^*$-extension of $C^*(E)$ by the stable ideal $\sk\otimes C^*(F)$, which is an essential ideal.
\end{enumerate}
See \Sref{Sec3} and \Sref{Sec4}. Due to the lack of multiplicative structure in $E$ and $F$, establishing $(i)$ requires an effort, and it is a result rather than a definition. The motivating example of such an extension arises when $F=\bbc\mbox{ and }E=C^*(\bbz)_{(n)}$, where
\begin{IEEEeqnarray}{lCl}\label{Fejer}
C^*(\bbz)_{(n)}:=\{f\in C(S^1):\widehat{f}(k)=0\mbox{ for }|k|\ge n\}
\end{IEEEeqnarray}
is the Fej\'er-Riesz operator system investigated in \cite{CS,S0}, which is dual to the operator system of $n\times n$ Toeplitz matrices \cite{F1}. In this case, we obtain a unital operator system $X\subseteq\mathcal{B}(\ell^2(\bbn))$ whose generated $C^*$-algebra is the Toeplitz algebra $\mathcal{T}$. At the $C^*$-algebra level, this yields the usual Toeplitz short exact sequence.
\smallskip

We now discuss several applications of our construction. In \cite{A2}, Arveson introduced the notion of \emph{hyperrigidity} for operator systems. An operator system $X$ is said to be hyperrigid in the $C^*$-algebra $C^*(X)$ that it generates, if for every nondegenerate representation $\pi:C^*(X)\to\mathcal{B}(H)$, the UCP map $\pi|_X$ has the \emph{unique extension property} (UEP). An irreducible representation with the UEP is called a boundary representation.

Proving hyperrigidity is often subtle, as it requires controlling all $\ast$-representations. Arveson's motivation arose from his earlier work \cite{A}, in which he introduced boundary representations and posed the question of whether every operator system admits sufficiently many boundary representations. He subsequently resolved this problem in the separable case \cite{A1}, while the general case was settled by Davidson and Kennedy \cite{DK}.

We show that a Toeplitz type extension $X(F;E,P)$ is hyperrigid if and only if both the ambient operator systems $E,F$ are hyperrigid (\Tref{first main result of hyperigidity}). This is achieved by decomposing all $\ast$-representations of the corresponding $C^*$-extension in terms of those for its ambient $C^*$-algebras (see \Pref{hyperlemma1}) and investigating all its irreducible representations (see \Pref{hyperrigidity irreducible repn}), together with the unique extension property in \Sref{Sec5.2} (in particular, \Tref{0th main result of hyperrigidity}). This yields a large class of hyperrigid operator systems through iteration: at each step, we replace $F$ by $X(F;E,P)$ and repeat the construction.

We introduce the notion of quantum double suspension of unital operator systems in \Sref{Sec5.3} as a special case of Toeplitz type extensions, taking $E$ to be the Fej\'er--Riesz operator system $C^*(\bbz)_{(n)}$ \Eref{Fejer}, for any fixed $n\ge 2$. By iterating the extension process, we obtain a sequence of hyperrigid operator systems whose generated $C^*$-algebras are precisely the Hong--Szyma\'nski iterated QDS. In particular, taking $F=C^*(\bbz)_{(n)}$ yields a sequence of hyperrigid operator systems whose generated $C^*$-algebras are the Vaksman--Soibelman odd-dimensional quantum spheres $C(S_q^{2\ell+1}),\,0<q<1$. Similarly, taking $F=\bbc\oplus\bbc$, the $C^*$-algebra of a two-point space, yields a sequence of hyperrigid operator systems whose generated $C^*$-algebras are the even-dimensional quantum spheres $C(S_q^{2\ell}),\,0<q<1$ (see \Yref{for fruitful remark}). Our investigation extends the pair $(C^*(\bbz)_{(n)},C(S^1))$, considered in \cite{CS}, to a broad class of important quantum spaces. We believe that this perspective opens up several intriguing directions for future research, such as the $K$-theory of operator systems \cite{S1,S2}.

Arveson also formulated the hyperrigidity conjecture \cite{A2} in an attempt to connect hyperrigidity with the noncommutative Choquet boundary. The conjecture is recently disproved by Bilich and Dor-On \cite{BD} in the noncommutative situation and by Bilich and Scherer \cite{BS} in the commutative situation, see \cite{B,CL,S,T} for more counterexamples. Our construction sheds new light on this conjecture (see Sections \ref{Sec5.4}, \ref{Sec6.2}). In \Sref{Sec5.4}, we show that every counterexample $F\subseteq\mathcal{B}(H_F)$ to the hyperrigidity conjecture, finite- or infinite-dimensional and in the commutative or noncommutative situation, gives rise to a sequence of counterexamples (see \Pref{1111} and \Yref{11111}).

In \Sref{Sec6.1}, we identify the \v{S}ilov ideal and the $C^*$-envelope of general Toeplitz type extensions (see \Pref{Silov ideal claim} and \Tref{Main theorem 1}). In \Sref{Sec6.2}, we revisit hyperrigidity and the hyperrigidity conjecture. We derive a $K$-theoretic condition on $C^*_{env}(F)$ (see \Tref{proposition for sequence of distinct elements}) under which the sequence arising from $F$ in \Pref{1111} contains infinitely many pairwise non-isomorphic counterexamples. In particular, Scherer's counterexample \cite{S} provides an instance of $F$ satisfying our condition and this concretely produces infinitely many pairwise non-isomorphic counterexamples to the hyperrigidity conjecture (see towards the end of \Sref{Sec6.2}).

We also observe that the validity of the conjecture for any function system in the commutative situation \cite{DK1} is linked to that of the conjecture for an infinite family of operator systems, obtained through successive extensions, whose generated $C^*$-algebras are noncommutative. If some member of this ``noncommutative'' family fails the conjecture, so does the commutative function system. Thus we obtain a procedure to generate more counterexamples in the commutative setting. This also illustrates the potential difficulty of obtaining a general characterization of the failure of the conjecture.

Finally, in the same work \cite{CS}, Connes and van Suijlekom also introduced an invariant of operator systems, known as the \emph{propagation number}, which quantifies the extent to which an operator system fails to be a $C^*$-algebra. We conclude the article by computing this invariant in \Sref{Sec7} for Toeplitz type extensions $X(F;E,P)$ (see \Tref{main result2}). It turns out that the class of Toeplitz type extensions contains several subclasses, each with a constant propagation number (see \Yref{main result22} and \ref{main result3}).

\medskip

\noindent\textbf{\textit{Acknowledgement.}} S. Guin thanks Partha Sarathi Chakraborty for clarifying a doubt on quantum spheres.


\section{Preliminaries}\label{Sec2}

In this section, we recall a few preliminary facts that are essential in the context of this article. For general results, we refer to \cite{ER,P,Pi}.

\begin{dfn} 
A unital (concrete) operator system $E$ is a $\ast$-closed unital linear subspace of $\mathcal{B}(H)$ for some Hilbert space $H$. It is called complete if it is additionally closed.
\end{dfn}
The celebrated Choi--Effros theorem shows that any abstract operator system is completely order isomorphic to a concrete operator system \cite{CE}.
\begin{dfn}[\cite{H,CS}]\label{envelope dfn}
A $C^*$-extension of a unital operator system $E$ consists of a pair $(A, \la)$, where $A$ is a unital $C^*$-algebra and $\la : E \rightarrow A$ is a unital complete order injection generating $A$ (i.e., $C^*(\la(E)) = A$).

The $C^*$-envelope of $E$ is a $C^*$-extension $( C^*_{env}(E),i_E)$ equipped with the universal property that for any $C^*$-extension $(A,\la)$, there exists a surjective $*$-homomorphism $\pi : A \longrightarrow C^*_{env}(E)$ such that the following diagram commutes:
\[
\begin{tikzcd}
    & A \arrow[d, "\pi", two heads] \\
    E \arrow[ur, "\la"] \arrow[r, "i_E"'] & C^*_{env}(E)
\end{tikzcd}
\]
\end{dfn}

\begin{dfn}[\cite{H}]
Let $E$  and $E_1$ be a unital operator systems such that there exists a unital complete order isomorphism $\phi: E\to E_1$. Suppose $(A,\la)$ and $(A_1,\la_1)$ are $C^*$-extensions of $E$ and $E_1$, respectively. We say that $(A,\la)$ and $(A_1,\la_1)$ are equivalent, denoted by $(A,\la)\sim (A_1,\la_1)$, if there exists a unital $\ast$-isomorphism $\widehat\phi: A\to A_1$ that makes the following diagram commute:
\[
\begin{tikzcd}
A \arrow[r, "\widehat\phi"] & A_1 \\
E \arrow[u, "\la"] \arrow[r, "\phi"'] & E_1 \arrow[u, "\la_1"']
\end{tikzcd}
\]
\end{dfn}

\begin{dfn}[\cite{A,H}]
For a unital operator system $E$ and a $C^*$-extension $(A,\lambda)$ of $E$, a boundary ideal of $A$ relative to $E$ is a closed two-sided ideal $I\subseteq A$ such that the quotient map $q:A\to A/I$ is a complete isometry on $\lambda(E)\subseteq A$. A boundary ideal of $A$ relative to $E$ is called the \v{S}ilov ideal of $A$ relative to $E$ if it contains every other boundary ideal of $A$ relative to $E$.
\end{dfn}
The existence and uniqueness of the $C^*$-envelope were established in complete generality by Hamana \cite{H}, without addressing Arveson's problem concerning boundary representations. Much later, Arveson revisited his original problem \cite{A1}, which was finally solved by Davidson and Kennedy in full generality \cite{DK}.

Concretely, for any operator system $E \subseteq B(H)$, considering the $C^*$-extension $C^*(E) \subseteq B(H)$ yields $C^*_{env}(E) \cong C^*(E)/\ker(\pi)$, where $\ker(\pi)$ (see \Dref{envelope dfn}) corresponds precisely to Arveson's \v{S}ilov ideal (\cite{A}, \cite[Sec. 4.1]{BL}).
\smallskip

For a unital operator system $E$ and a $C^*$-extension $(A, \lambda)$ of $E$, define the unital operator system
\[
(\lambda(E))^{\circ n} = \overline{\text{span}\{a_1 \dots a_n : a_i \in \lambda(E)\}} \subseteq A
\]
for each $n \in \mathbb{N}$.
\begin{dfn}[\cite{CS}]\label{the reqd dfn}
Let $E$ be a unital operator system, and let $(C^*_{env}(E), i_E)$ be the $C^*$-envelope of $E$. The propagation number of $E$ is defined as 
\[
\operatorname{prop}(E) := \min\{n \in \mathbb{N} : (i_E(E))^{\circ n} = C^*_{env}(E)\}.
\]
If no such $n$ exists, we set $\operatorname{prop}(E) := \infty$.
\end{dfn}

We recall the following two results, which will be used in this article.
\begin{lmma}[\text{\cite[Proposition 2.6]{CS}}]\label{P1}
Let $\varphi : E \to F$ be a linear map between operator systems.
\begin{enumerate}[$(i)$]
\item If $\varphi$ is completely positive, then it is completely bounded with
$\|\varphi\|_{\mathrm{cb}} = \|\varphi\| = \|\varphi(1)\|$.
\item A unital map $\varphi$ is completely positive if and only it is completely contractive.
\item A unital map $\varphi$ is a complete order injection if and only if it is completely isometric.
\end{enumerate}
\end{lmma}

\begin{thm}[\text{\cite[Theorem 4.6]{KPTT}}]\label{P2}
Let $\varphi : E \to F$ and $\psi: E'\to F'$ be unital completely positive maps (resp., unital complete order injections) between unital operator systems. Then $\,\varphi\otimes_{\text{min}} \psi: E\otimes_{\text{min}}  E'\to F\otimes_{\text{min}}  F'$ is a unital completely positive map (resp., unital complete order injection).
\end{thm}

Recall that if $\sigma:E\to F$ is a surjective UCP, then the induced map $\widetilde\sigma:E/\ker(\sigma)\to F$, defined by $\widetilde\sigma(e+\ker(\sigma)):=\sigma(e)$, is well-defined, bijective, and UCP \cite{KPTT1}. However, this does not guarantee that $\widetilde\sigma^{-1}$ is also UCP.

\begin{dfn}[\cite{KPTT1}]
Let $\sigma:E\to F$ be a surjective UCP. We call $E$ an operator-system extension of $F$ if $\,\widetilde\sigma:E/\ker(\sigma)\to F$ induced by $\sigma$ is a complete order isomorphism, or in other words, $\sigma$ is a complete quotient map.
\end{dfn}
The following result will be used in \Sref{Sec5.2}.
\begin{thm}[\text{\cite[Theorem 3.18]{P}}]\label{Multiplicative domain theorem}
Let \(\mathcal{A}\) and \(\mathcal{B}\) be unital \(C^*\)-algebras, and let \(\phi: \mathcal{A} \to \mathcal{B}\) be unital completely positive. We have the following:
\begin{enumerate}[$(i)$]
\item \(\{a \in \mathcal{A}: \phi(a)^*\phi(a) = \phi(a^*a)\} = \{a \in \mathcal{A}: \phi(ba) = \phi(b)\phi(a) \text{ for all } b \in \mathcal{A}\}\) is a subalgebra of \(\mathcal{A}\), and \(\phi\) is a homomorphism when restricted to this set.
\item \(\{a \in \mathcal{A}: \phi(a)\phi(a)^* = \phi(aa^*)\} = \{a \in \mathcal{A}: \phi(ab) = \phi(a)\phi(b) \text{ for all } b \in \mathcal{A}\}\) is a subalgebra of \(\mathcal{A}\), and \(\phi\) is a homomorphism when restricted to this set.
\item \(\{a \in \mathcal{A}: \phi(a)^*\phi(a) = \phi(a^*a) \text{ and } \phi(a)\phi(a)^* = \phi(aa^*)\} = \{a \in \mathcal{A}: \phi(ab) = \phi(a)\phi(b) \text{ and } \phi(ba) = \phi(b)\phi(a) \text{ for all } b \in \mathcal{A}\}\) is a \(C^*\)-subalgebra of \(\mathcal{A}\), and \(\phi\) is a \(*\)-homomorphism when restricted to this set.
\end{enumerate}
We denote multiplicative domain of $\phi$ by $$\mathrm{MD}(\phi)= \{a \in \mathcal{A}: \phi(ab) = \phi(a)\phi(b) \text{ and } \phi(ba) = \phi(b)\phi(a) \text{ for all } b \in \mathcal{A}\}.$$
\end{thm}


\section{The class of unital Operator Systems}\label{Sec3}

In this section, we introduce the class of operator systems that we will study in this article. All tensor products of operator systems considered in this article are minimal.
\smallskip

\noindent\textbf{\textit{The family $\mathrm{Op}_{C^*}$}:} Consider the family $\mathrm{Op}_{C^*}$ consisting of a tuple $(A,H,P)$, where $A$ is a unital $C^*$-algebra, $\pi:A\to\mathcal{B}(H)$ is a faithful representation on a separable Hilbert space $H$, and $P\in\mathcal{B}(H)$ is an infinite-dimensional projection such that
\begin{itemize}
\item $[P,\pi(a)]\in \sk(H)$ for all $a\in A$;
\item If $a\in A$ is such that $P\pi(a)P\in\sk(PH)$, then $a=0$.
\end{itemize}
Together, these two conditions imply that $A$ has a faithful representation in the Calkin algebra $\mathcal{Q}(PH)$. Note that the second condition is equivalent to saying $P\pi(A)P\cap\sk(PH)=\{0\}$ and the map $a\mapsto P\pi(a)P$ is injective.
\begin{dfn}
An element $(A,H,P)\in\mathrm{Op}_{C^*}$ is called degenerate when $P=\mathrm{id}$, and otherwise it is called non-degenerate.
\end{dfn}

Notice that the degenerate case simply demands $\pi(A)\cap\mathcal{K}(H)=\{0\}$.

\begin{xmpl}\rm
\begin{enumerate}[$(i)$]
\item Consider $\pi:C(S^1)\to\mathcal{B}(L^2(S^1))$ the standard representation by multiplication of functions and the Hardy projection $P:L^2(S^1)\cong\ell^2(\mathbb{Z})\to\ell^2(\mathbb{N})$. The tuple $(C(S^1),L^2(S^1),P)$ is a non-degenerate element in $\mathrm{Op}_{C^*}$.
\item Let $A$ be a unital $C^*$-algebra with a faithful representation $\pi_1:A\to\mathcal{B}(H_1)$ on a separable Hilbert space $H_1$ such that $\pi_1(A)\cap\mathcal{K}(H_1)=\{0\}$. Let $\pi_2:A\to\mathcal{B}(H_2)$ be any unital representation on a separable Hilbert space $H_2$ and consider $\pi:=\pi_1\oplus\pi_2:A\to\mathcal{B}(H_1\oplus H_2)$. Then $(A,H_1\oplus H_2,P)\in\mathrm{Op}_{C^*}$, where $P=1_{H_1}\oplus\,0$. This scheme produces many non-degenerate elements in $\mathrm{Op}_{C^*}$ from degenerate ones.
\item Let $A$ be a simple unital infinite-dimensional $C^*$-algebra with a faithful representation $\pi:A\rightarrow\mathcal{B}(H)$ on a separable Hilbert space $H$. Then $\pi(A)\cap\mathcal{K}(H)=\{0\}$, which gives a degenerate element $(A,H,\mathrm{id})\in\mathrm{Op}_{C^*}$. To further get a non-degenerate element in $\mathrm{Op}_{C^*}$, apply the construction described in $(ii)$.
\end{enumerate}
\end{xmpl}
It is clear that the family $\mathrm{Op}_{C^*}$ is large.

\begin{rmrk}
$\mathrm{Op}_{C^*}$ can be made into a category. A morphism $(\phi,V):(A,H_A,P_A)\longrightarrow(B,H_B,P_B)$
consists of a unital $\ast$-homomorphism $\phi:A\rightarrow B$ and a bounded linear map $V:H_A\rightarrow H_B$ such that
\begin{enumerate}[$(i)$]
\item $V\pi_A(a)-\pi_B(\phi(a))V\in\mathcal{K}(H_A,H_B),\,\, \forall\,a\in A$;
\item $VP_A-P_BV\in\mathcal{K}(H_A,H_B)$.
\end{enumerate}
The composition is defined by $(\psi,W)\circ(\phi,V):=(\psi\circ\phi\,,\,WV)$.
\end{rmrk}

\noindent\textbf{\textit{The family $\mathrm{Op}_{\mathrm{OS}}$}:} Let $E\subseteq\mathcal{B}(H)$ be a unital operator system with $H$ a separable Hilbert space. Consider the $C^*$-algebra $C^*(E)\subseteq\mathcal{B}(H)$ generated by $E$. Suppose there exists an infinite-dimensional projection $P\in\mathcal{B}(H)$ such that
\begin{itemize}
\item $[P,e]\in \sk(H)$ for all $e\in E$;
\item If $c\in C^*(E)$ is such that $PcP\in\sk(PH)$, then $c=0$.
\end{itemize}
Consider the family $\mathrm{Op}_{\mathrm{OS}}$ consisting of such a tuple $(E,H,P)$. Like in the case of $\mathrm{Op}_{C^*}$, we call an element $(E,H,P)\in\mathrm{Op}_{\mathrm{OS}}$ degenerate when $P=\mbox{id}$, and otherwise non-degenerate. Observe that $\mathrm{Op}_{C^*}\subseteq\mathrm{Op}_{\mathrm{OS}}$.

\begin{xmpl}\label{crucial example}\rm
\begin{enumerate}[$(i)$]
\item Consider the Fej\'er-Riesz operator system $E=C^*(\bbz)_{(n)}:=\{f\in C(S^1):\widehat{f}(k)=0\mbox{ for }|k|\ge n\}$ studied in \cite{CS,S0}, which is dual to the operator system of $n\times n$ Toeplitz matrices. We have $E\subseteq\mathcal{B}(\ell^2(\bbz))$. Take $P:L^2(S^1)\cong\ell^2(\bbz)\to\ell^2(\bbn)$, the Hardy projection. Then $(E,\ell^2(\bbz),P)$ is a non-degenerate element in $\mathrm{Op}_{\mathrm{OS}}$. Note that $C^*(C^*(\bbz)_{(n)})=C(S^1)$ and $PC(S^1)P\,\cap\,\sk(\ell^2(\bbn))=\{0\}$.
\item Let $\,\theta\,$ be an irrational number and consider the finite-dimensional unital operator system $\mathcal{I}_\theta:=\mathrm{sp}\{1,u,v,u^*,v^*\}\subseteq\mathcal{B}(\ell^2(\bbz^2))$, where $u,v$ are any two unitaries acting on $\ell^2(\bbz^2)$ satisfying $uv=e^{2\pi i\theta}vu$. The $C^*$-algebra it generates is the noncommutative torus $\mathcal{A}_\theta$, which is simple. Then $(\mathcal{I}_\theta,\ell^2(\bbz^2),\mathrm{id})$ is a degenerate element in $\mathrm{Op}_{\mathrm{OS}}$.
\end{enumerate}
\end{xmpl}
The above examples show that $\mathrm{Op}_{C^*}\subsetneq\mathrm{Op}_{\mathrm{OS}}$.
\begin{rmrk}
$\mathrm{Op}_{\mathrm{OS}}$ can also be made into a category. A morphism $
(\phi,V):(E,H_E,P_E)\longrightarrow(F,H_F,P_F)$ consists of a UCP $\phi:E\to F$ and a bounded linear map $V:H_E\rightarrow H_F$ such that
\begin{enumerate}[$(i)$]
\item $Ve-\phi(e)V\in\mathcal{K}(H_E,H_F),\,\,\forall\,e\in E$;
\item $VP_E-P_FV\in\mathcal{K}(H_E,H_F)$.
\end{enumerate}
The composition is defined by $(\psi,W)\circ(\phi,V):=(\psi\circ\phi\,,\,WV)$. Then, $\mathrm{Op}_{C^*}$ is a subcategory of $\mathrm{Op}_{\mathrm{OS}}$.
\end{rmrk}

Consider the canonical map
\[
\Phi:\mathrm{Op}_{\mathrm{OS}}\longrightarrow\mathrm{Op}_{C^*}\,,\qquad (E,H_E,P)\longmapsto (C^*(E),H_E,P).
\]
For the well-definedness, we only need to check that $[P,e]\in\sk(H_E)$ for all $e\in E$ implies that $[P,x]\in \sk(H_E)$ for all $x\in C^*(E)$. For this, consider
\[
\mathscr{A}:=\{T\in\mathcal{B}(H_E):[P,T]\in\mathcal{K}(H_E)\},
\]
and observe that $\mathscr{A}$ is a unital $C^*$-subalgebra of $\mathcal{B}(H_E)$ containing $E$. Consequently, $C^*(E)\subseteq\mathscr{A}$. Since $C^*(C^*(E))=C^*(E)$ for any unital operator system $E\subseteq\mathcal{B}(H)$, we see that the map $\Phi$ is idempotent.
\smallskip

\noindent\textbf{\textit{Our construction}:} Let $\mathrm{OS}_1$ denote the category of unital operator systems (without any restriction on dimensionality) with UCP maps as morphisms. Define the map
\[
\Psi:\mathrm{OS}_1\times\mathrm{Op}_{\mathrm{OS}}\longrightarrow\mathrm{OS}_1\,,\qquad (F,(E,H_E,P))\longmapsto X(F;E,P),
\]
where
\begin{IEEEeqnarray}{lCl}\label{our construction}
X(F;E,P):= \sk(PH_E)\otimes_{\min} F + PEP\otimes \mathbb{C}\cdot 1_{H_{F}}.
\end{IEEEeqnarray}
Then $X(F;E,P)$ is a unital operator system in $\mathcal{B}(PH_E\otimes H_F)$ with the unit $P\otimes 1_{H_F}$.

If we restrict the class of morphisms $(\phi,V)$ between objects $(E,H_E,P)$ and $(E^\prime, H_{E^\prime},P^\prime)$ in $\mathrm{Op}_{\mathrm{OS}}$ by demanding that $V:H_E\to H_{E^\prime}$ be unitary, then $\Psi$ becomes a functor.

\begin{lmma}\label{lemma used in helpful lemma 1}
We have $(\sk(PH_E)\otimes_{\min} F)\cap(PEP\otimes \mathbb{C}\cdot 1_{H_F})=\{0\}$ for any $E,F,\mbox{ and }P$.
\end{lmma}
\begin{proof}
Let $w\in S(F)$ be a state on $F$ and write $k=PeP\otimes 1$ with $k\in \sk(PH_E)\otimes_{\min} F$ and $e\in E$. Consider the UCP $S_w:=\mathrm{id}_{\mathcal{B}(PH_E)}\otimes_{\text{min}} w$ (and hence complete contraction by Lemma \ref{P1}), and use the fact that $PeP=S_w(k)\in \sk(PH_E)$ implies $e=0$.
\end{proof}

\begin{lmma}\label{lemma used in helpful lemma 2}
\begin{enumerate}[$(i)$]
\item The map $\varphi:E \to X(F;E,P)$, defined by $\varphi(e)=PeP\otimes 1_{H_F}$, is an injective unital completely positive map.
\item The map $\sigma:X(F;E,P)\to E$ given by $k+PeP\otimes1_{H_F}\mapsto e$ is well-defined and surjective.
\end{enumerate}
\end{lmma}
\begin{proof}
This easily follows from the definition of $\mathrm{Op}_{\mathrm{OS}}$.
\end{proof}

\begin{ppsn}\label{qqq}
We have $X(F;E,P)$ is closed in $\mathcal{B}(PH_E\otimes H_F)$ whenever $E$ is closed, irrespective of whether $F\subseteq\mathcal{B}(H_F)$ is closed or not.
\end{ppsn}
\begin{proof}
We have seen that $[P,e]\in\sk(H_E)$ for all $e\in E$ implies that $[P,e]\in\sk(H_E)$ for all $e\in C^*(E)$. Consider the map $$\tau: C^*(E)\to\big(\mathcal{B}(PH_E)\otimes_{\min} C^*(F)\big)/\big(\sk(PH_E)\otimes C^*(F)\big)$$ defined by $\tau(c)=PcP\otimes 1_{H_{F}}+ \sk(PH_E)\otimes C^*(F)$. Then, $\tau$ is an injective $*$-homomorphism of $C^*$-algebras. Therefore, $\tau$ is an isometry, and hence $$\norm{PcP\otimes 1_{H_F}+\sk(PH_E)\otimes C^*(F)}_{\mbox{\tiny Quotient}}=\norm{c}_{C^*(E)}.$$ This implies that $$\norm{k+PcP\otimes 1_{H_F}}_{\mathcal{B}(PH_E)\otimes_{\text{min}} C^*(F)}\geq \norm{c}_{C^*(E)},$$ for all $k\in \sk(PH_E)\otimes C^*(F)$ and $c\in C^*(E)$. Now let $(z_n)_{n\in\mathbb{N}}$ be a sequence with $z_n=k_n+Pc_nP\otimes 1_{H_F}\in X(F;E,P)$ such that $z_n\to z\in \mathcal{B}(PH_E)\otimes_{\text{min}} C^*(F)$. Then,
\begin{eqnarray*}
\norm{c_n-c_m}_{C^*(E)} &\leq& \norm{k_n-k_m+P(c_n-c_m)P\otimes 1_{H_F}}_{\mathcal{B}(PH_E)\otimes_{\text{min}} C^*(F)}\\
&=& \norm{z_n-z_m}_{\mathcal{B}(PH_E)\otimes_{\text{min}} C^*(F)}.
\end{eqnarray*}
This shows that $\{c_n\}_{n\in \mathbb{N}}$ is a Cauchy sequence in $E$. Since $E$ is assumed to be closed, this implies that $c_n\to c$ for some $c\in E$. By Lemma \ref{lemma used in helpful lemma 2}, since $\varphi: E\to  X(F;E,P)$ is continuous, we have $\varphi(c_n)= Pc_nP\otimes 1_{H_F}\to PcP\otimes 1_{H_F}=\varphi(c)\in X(F;E,P)$. Therefore, $k_n=z_n-Pc_nP\otimes 1_{H_F}\to z-PcP\otimes 1_{H_F}\in \mathcal{B}(PH_E)\otimes_{\text{min}} C^*(F)$. Since $k_n\in \sk(PH_E)\otimes_{\min} F$, we get $k:=z-PcP\otimes 1_{H_F}\in \sk(PH_E)\otimes_{\min} F$ by the closedness of $\sk(PH_E)\otimes_{\min} F$ (irrespective of whether $F$ is closed or not). This shows that $z=k+PcP\otimes 1_{H_F}\in X(F;E,P)$, which completes the proof.
\end{proof}

\begin{ppsn}\label{1}
Let $E$ be a unital $C^*$-algebra with unital faithful representation $\pi_E:E\to\mathcal{B}(H_E)$ such that $\pi_E(E)\,\cap\,\sk(H_E)=\{0\}$. In other words, $(\pi_E(E),H_E,\mathrm{id})\in\mathrm{Op}_{\mathrm{C^*}}$ is a degenerate element. Let $F\subseteq\mathcal{B}(H_F)$ be a unital operator system. Then the map $\sigma:X(F;\pi_E(E),\mathrm{id})\to E$, given by $k+\pi_E(e)\otimes1_{H_F}\mapsto e$, is UCP.
\end{ppsn}
\begin{proof}
Denote $X:=X(F;\pi_E(E),\mathrm{id})$. Consider the quotient homomorphism $q:\mathcal{B}(H_E)\longrightarrow\mathcal{Q}(H_E)=\mathcal{B}(H_E)/\sk(H_E)$. Since $\pi_E(E)\cap\sk(H_E)=\{0\}$, we have that $q\big|_{\pi_E(E)}$ is an injective unital $\ast$-homorphism, and hence a UCP. Therefore,
\[
q\otimes_{\min}\mathrm{id}:\mathcal{B}(H_E)\otimes_{\min}F\longrightarrow\mathcal{Q}(H_E)\otimes_{\min}F
\]
is a UCP by \Tref{P2}. Since $q^{-1}|_{q(\pi_E(E))}$ is also a UCP with domain the $C^*$-algebra $q(\pi_E(E))$, and $E\cong E\otimes\bbc$ as $C^*$-algebras, we have $\sigma=(q^{-1}|_{q(\pi_E(E))}\otimes_{\min}\mbox{id})\circ(q\otimes_{\min}\mathrm{id})|_X$is a UCP.
\end{proof}

In the next section, we will be able to extend the above result to general elements of $\mathrm{Op}_{\mathrm{OS}}$.


\section{Toeplitz type Extension of Operator Systems}\label{Sec4}

Let $F\subseteq\mathcal{B}(H_F)$ be a unital operator system and consider $(E,H_E,P)$ in $\mathrm{Op}_{\mathrm{OS}}$. We have the unital operator system $X:=X(F;E,P)$ defined in \Eref{our construction}. In this section, we compute the $C^*$-algebra generated by $X$ in $\mathcal{B}(PH_E\otimes H_F)$, and extend \Pref{1} to $\mathrm{Op}_{\mathrm{OS}}$ to define Toeplitz type extension of operator systems. We denote by $(C^*(F),i)$ a $C^*$-extension of $F$.

\begin{lmma}\label{C*(X) contains that}
For all $n\in \mathbb{N}$, we have $\sk(PH_E)\otimes_{\min} F^{on}\subseteq X^{on}$ and $PE^{on}P\otimes 1_{H_F}\subseteq X^{on}$.
\end{lmma} 
\begin{proof}
To see this, let $\{\eta_i\}_{i\in \mathbb{N}}$ be an orthonormal basis of $PH_E$ and let $f_N$ be the orthogonal projection onto $\text{span}\{\eta_i\}_{i=1}^N$. Fix $k\in \sk(PH_E)$ and let $a_1,\dots, a_n\in F$. Then $k\otimes a_1, f_N\otimes a_2, \dots, f_N\otimes a_n \in \sk(PH_E)\otimes_{\min} F\subseteq X$, and thus $kf_N\otimes a_1\dots a_n\in X^{on}$. Now, we use the fact that $\norm{kf_N-k}_{\mathcal{B}(PH_E)}\to 0$ as $N\to \infty$ (because $f_N\to\mathrm{id}$ strongly). So,
\[
\norm{(kf_N\otimes a_1\dots a_n) - (k \otimes a_1\dots a_n)}_{\mathcal{B}(PH_E \otimes H_F)} = \norm{kf_N-k}_{\mathcal{B}(PH_E)}\cdot \norm{a_1\dots a_n}_{\mathcal{B}(H_F)}\to 0
\]
as $N\to \infty$. Thus, $k\otimes a_1\dots a_n \in X^{on}$, which implies that $k\otimes \text{span}\{a_1\dots a_n : a_i\in F\}\subseteq X^{on}$. Therefore, we obtain $k\otimes F^{on}\subseteq X^{on}$ for every $k\in \sk(PH_E)$. Hence, $\sk(PH_E)\odot F^{on}\subseteq X^{on}$, which further implies that
\[
\sk(PH_E)\otimes_{\min} F^{on} = \overline{\sk(PH_E)\odot F^{on}}^{\,\norm{\cdot}_{\mathcal{B}\left(PH_E \otimes H_F\right)}}\subseteq X^{on}.
\]

We now show that $PE^{on}P\otimes 1_{H_F}\subseteq X^{on}$. Let $e_1,\dots,e_n\in E$. Then $Pe_iP\otimes 1_{H_F}\in X$ for all $i$, and hence $(Pe_1P\otimes 1_{H_F})\dots(Pe_nP\otimes 1_{H_F})\in X^{on}$. However,
\[
(Pe_1P\otimes 1_{H_F})\dots(Pe_nP\otimes 1_{H_F})=Pe_1\dots e_nP\otimes 1_{H_F}+ k\otimes 1_{H_F}
\]
for some $k\in \sk(PH_E)$ because $[P,e_i]\in\mathcal{K}(H_E)$. From the previous argument, we know that $k\otimes 1_{H_F}\in X^{on}$, which implies $Pe_1\dots e_nP\otimes 1_{H_F}\in X^{on}$. Thus, $P\,\text{span}\{e_1\dots e_n : e_i \in E\}P\otimes 1_{H_F}\subseteq X^{on}$, and consequently $PE^{on}P\otimes 1_{H_F}\subseteq X^{on}$.
\end{proof}

\begin{lmma}\label{Xon}
$X^{on}=\overline{\sk(PH_E)\otimes_{\min} F^{on} + PE^{on}P\otimes1_{H_F}}\,\,\mbox{ for all } n\in \mathbb{N}$.
\end{lmma}
\begin{proof}
Observe that for all $n\in \mathbb{N}$, we have $X^{on}\subseteq \overline{\sk(PH_E)\otimes_{\min} F^{on} + PE^{on}P\otimes1_{H_F}}\,$. The case of $n=1$ is by definition and the case of $n=2$ is easy to verify, and the general case follows from induction. Combine this with Lemma \ref{C*(X) contains that}. 
\end{proof}

\begin{ppsn}\label{C*(X)}
We have $C^*(X)=\overline{\sk(PH_E)\otimes C^*(F) + PC^*(E)P\otimes1_{H_F}}$.
\end{ppsn}
\begin{proof}
By \Lref{Xon}, we have $X^{on}\subseteq \overline{\sk(PH_E)\otimes C^*(F) + PC^*(E)P\otimes1_{H_F}}$ for every $n\in \mathbb{N}$. Thus, $C^*(X)\subseteq \overline{\sk(PH_E)\otimes C^*(F) + PC^*(E)P\otimes1_{H_F}}$. Therefore, we need to show
\[
\sk(PH_E)\otimes C^*(F)+ PC^*(E)P\otimes 1_{H_F}\subseteq C^*(X).
\]
From \Lref{C*(X) contains that}, we have $\{k\}\otimes F^{on}\subseteq X^{on}\subseteq C^*(X)$ for every $n\in \mathbb{N}$ and $k\in \sk(PH_E)$. Thus, $\{k\}\otimes\bigcup_{n=1}^{\infty} F^{on}\subseteq C^*(X)$ for every $k\in \sk(PH_E)$. Hence, we get $\{k\}\otimes C^*(F)=\{k\}\otimes \overline{\bigcup_{n=1}^{\infty} F^{on}}\subseteq C^*(X)$ for all $k\in \sk(PH_E)$. Therefore, $\sk(PH_E)\odot C^*(F)\subseteq C^*(X)$, which implies $\sk(PH_E)\otimes C^*(F)\subseteq C^*(X)$. Similarly, since $PE^{on}P\otimes 1_{H_F}\subseteq X^{on}\subseteq C^*(X)$ for every $n\in \mathbb{N}$, it follows that $PC^*(E)P\otimes 1_{H_F}=P\left(\overline{\bigcup_{n=1}^{\infty}E^{on}}\right)P\otimes 1_{H_F}\subseteq C^*(X)$. Therefore, we have $\sk(PH_E)\otimes C^*(F),\, PC^*(E)P\otimes 1_{H_F}\subseteq C^*(X)$, concluding the proof.
\end{proof}

Next, we prove the following general lemma, which will be useful throughout the article.

\begin{lmma}\label{helpful lemma}
Consider an element $(E,H_E,P)\in\mathrm{Op}_{\mathrm{OS}}$. Let $D$ be a unital $C^*$-algebra in $\mathcal{B}(H_D)$ and consider $$\mathcal{C}(X_D):= \sk(PH_E)\otimes D+ PC^*(E)P\otimes 1_{H_D},$$ a unital operator system in $\mathcal{B}(PH_E\otimes H_D)$. Then, we have the following:
\begin{enumerate}[$(i)$]
\item The map $\varphi: C^*(E)\to \mathcal{C}(X_D)$, defined by $\varphi(e):= PeP\otimes 1_{H_D}$, is an injective unital completely positive map.
\item $(\sk(PH_E)\otimes D)\ \cap\ (PC^*(E)P\otimes 1_{H_D})=\{0\}$.
\item $\mathcal{C}(X_D)$ is a unital $C^*$-subalgebra of $\mathcal{B}(PH_E \otimes H_D)$.
\item $\sk(PH_E)\otimes D$ is an essential ideal of $\mathcal{C}(X_D)$, i.e. it has nonzero intersection with any nonzero ideal in $\mathcal{C}(X_D)$.
\end{enumerate}
\end{lmma}
\begin{proof}
The proof of $(i)$ follows from the definition of $\mathrm{Op}_{\mathrm{OS}}$, and that of $(ii)$ follows from \Lref{lemma used in helpful lemma 1}. The proof of part $(iii)$ is similar to the proof of \Pref{qqq}. Therefore, the only thing remaining is part $(iv)$.

Let $\{\eta_i\}_{i\in \mathbb{N}}$ be an orthonormal basis of $PH_E$ and $f_N$ be the orthogonal projection onto $\text{span}\{\eta_i\}_{i=1}^N$. We have $f_N\in \sk(PH_E)$ for all $N\in \mathbb{N}$, and thus $f_N\otimes 1_{D}\in \sk(PH_E)\otimes D\subseteq \mathcal{C}(X_D)$. Since $f_N\to\mbox{Id}$ strongly, we have the following:
\begin{eqnarray}\label{Fact1}
\norm{(f_N\otimes1_{H_D})k-k}_{\mathcal{B}(PH_E \otimes H_D)}\longrightarrow 0\quad as\,\,\,N\to \infty.
\end{eqnarray}
for all $k\in \sk(PH_E)\otimes D$, and
\begin{eqnarray}\label{Fact2}
\norm{(f_N PcP)(\eta)-PcP(\eta)}_{PH_E}\longrightarrow 0\quad as\,\,\,N\to \infty.
\end{eqnarray}
for all $c\in C^*(E)$ and $\eta \in PH_E$. Now, let $L$ be an ideal of $\mathcal{C}(X_D)$ such that $L\cap (\sk(PH_E)\otimes D)=\{0\}$. We show that $L=\{0\}$. Let $z=-k+PcP\otimes 1_{H_D}\in L$. Since $f_N\otimes 1_{H_D}\in \sk(PH_E)\otimes D$, we have $(f_N\otimes 1_{H_D})z\in L\cap (\sk(PH_E)\otimes D)=\{0\}$. This implies $(f_N\otimes 1_{H_D})(PcP\otimes 1_{H_D})=(f_N\otimes 1_{H_D})k$, and therefore $f_N PcP\otimes 1_{H_D}=(f_N\otimes 1_{H_D})k$ for all $N\in \mathbb{N}$. Therefore, by \Eref{Fact1} we obtain
\[
\norm{f_N PcP\otimes 1_{H_D}-k}_{\mathcal{B}(PH_E \otimes H_D)}\to 0 \quad \text{as } N\to \infty.
\]
Now, let $w\in S(D)$ be a state on $D$ and consider $S_w:=\mbox{id}_{\mathcal{B}(PH_E)}\otimes w: \mathcal{B}(PH_E)\otimes_{\min} D\to \mathcal{B}(PH_E)\otimes_{\min} \mathbb{C}\cong \mathcal{B}(PH_E)$. Observe that $S_w(T\otimes d)=w(d)T$, and so $S_w(\sk(PH_E)\odot D)\subseteq \sk(PH_E)$. Thus, $S_w(\sk(PH_E)\otimes D)\subseteq \sk(PH_E)$. Hence,
\[
f_N PcP=S_w(f_N PcP\otimes 1_{H_D})\xrightarrow{\text{\tiny$\norm{\cdot}_{\mathcal{B}(PH_E)}$}} S_w(k)=:k_0\in \sk(PH_E) \quad \text{as } N \to \infty.
\]
Therefore, for all $\eta\in PH_E$, we have $\norm{(f_N PcP)(\eta)-k_0(\eta)}_{PH_E}\to 0$ as $N\to \infty$. Combining this with \Eref{Fact2}, we obtain $PcP(\eta)=k_0(\eta)$ for all $\eta\in PH_E$, which implies $PcP=k_0\in \sk(PH_E)$. By hypothesis, this yields $c=0$, and hence $PcP\otimes 1_{H_D}=0$. Therefore,
\[
z=-k+PcP\otimes 1_{H_D}=-k\in \sk(PH_E)\otimes D\implies z\in L\cap(\sk(PH_E)\otimes D).
\]
Therefore, $z=0$, which proves that $L=\{0\}$.
\end{proof}

\begin{lmma}\label{sigma map used in prop number}
The map $\widetilde\sigma:\mathcal{C}(X_D)\to C^*(E)$, defined by $\widetilde\sigma(k+PcP\otimes 1)= c$, is a well defined, surjective, $\ast$-homomorphism with kernel $\,\ker(\widetilde\sigma)=\sk(PH_E)\otimes D$.
\end{lmma}
\begin{proof}
Straightforward verification using the fact $PC^*(E)P\cap \sk(PH_E)=\{0\}$.
\end{proof}

\begin{thm}\label{mainresult 0}
The $C^*$-algebra generated by $X:=X(F;E,P)$ in $\mathcal{B}(PH_E\otimes H_F)$ is $$C^*(X)=\sk(PH_E)\otimes C^*(F)+ PC^*(E)P\otimes 1_{H_F}$$ where the sum is linearly direct. Moreover, $\sk(PH_E)\otimes C^*(F)$ is an essential ideal of $C^*(X)$.
\end{thm}
\begin{proof} 
Apply \Pref{C*(X)} together with Lemma \ref{helpful lemma}.
\end{proof}

\begin{dfn}[Defn. 1.1, \cite{HZ}]\label{Hawkins Zacharias defn}
Let $A,B$ be unital $C^*$-algebras with faithful representations $\pi_A:A\to\mathcal{B}(H_A)$ and $\pi_B:B\to\mathcal{B}(H_B)$ respectively. A short exact sequence of the form
\[
0\longrightarrow\sk\otimes B\longrightarrow\mathscr{C}\longrightarrow A\longrightarrow 0
\]
is called a Toeplitz type $C^*$-extension if there exists an infinite-dimensional projection $P\in\mathcal{B}(H_A)$ such that
\begin{enumerate}[$(i)$]
\item $[P,a]\in\sk(H_A)$ for all $a\in A$;
\item $\mathscr{C}\cong \sk(PH_A)\otimes B+ PAP\otimes\bbc$;
\item $\sk(PH_A)\otimes B\cap PAP\otimes\bbc=\{0\}$.
\end{enumerate}
\end{dfn}

\begin{rmrk}
Hawkins and Zacharias restricted their attention to separable $C^*$-algebras as they were mainly interested in the metric aspect of spectral triple in the sense of Rieffel \cite{Rie1}. We do not need separability assumption on $C^*$-algebras in our article.
\end{rmrk}

\begin{crlre}\label{mainresult 01}
$C^*(X)$ is a Toeplitz type $C^*$-algebra extension of $C^*(E)$ by the stable ideal $\sk(PH_E)\otimes C^*(F)$, which is an essential ideal.
\end{crlre}
\begin{proof}
Taking $D=C^*(F)$ in \Lref{sigma map used in prop number}, \Tref{mainresult 0} gives us a short exact sequence of $C^*$-algebras
\[
0\longrightarrow \sk(PH_E)\otimes C^*(F)\longrightarrow C^*(X)\longrightarrow C^*(E)\longrightarrow 0
\]
satisfying all the requirements of \Dref{Hawkins Zacharias defn}.
\end{proof}

\begin{crlre}\label{2}
The map $\sigma:X:=X(F;E,P)\to E$ given by $k+PeP\otimes1_{H_F}\mapsto e$ is UCP for any unital operator system $F\subseteq\mathcal{B}(H_F)$ and any $(E,H_E,P)\in \mathrm{Op}_{\mathrm{OS}}\,$.
\end{crlre}
\begin{proof}
Taking $D=C^*(F)$ in \Lref{sigma map used in prop number}, we have $\sigma$ is the restriction of a surjective $\ast$-homomorphism $\widetilde\sigma:C^*(X)\to C^*(E)$ to the unital operator system $X$.
\end{proof}

\begin{crlre}
We have the following short exact sequence
\[
0\longrightarrow \sk(PH_E)\otimes_{\min}F\longrightarrow X(F;E,P)\overset{\sigma}{\longrightarrow} E\longrightarrow 0
\]
in the operator-system category.
\end{crlre}
\begin{proof}
By \Yref{2}, we have a UCP $\sigma:X(F;E,P)\to E$ with kernel $\ker(\sigma)=\sk(PH_E)\otimes_{\min} F$ and by \Lref{lemma used in helpful lemma 2} we have a UCP splitting map $\varphi:E\to X(F;E,P)$ with $\sigma\circ\varphi=\mbox{id}$. Consider the quotient map $q:X(F;E,P)\to X(F;E,P)/\ker(\sigma)$, where the quotient is equipped with quotient-operator-system structure \cite{KPTT1}. Thus, we have UCP $q\circ\varphi: E\to X(F;E,P)/\ker(\sigma)$, which is injective in our situation. We have the map $\psi:X(F;E,P)/\ker(\sigma)\to E$ induced by $\sigma$, which is a bijective UCP \cite{KPTT1}. The inverse of $\psi$ is the map $q\circ\varphi$, which proves $E \cong_{\mbox{\tiny coi}} X(F;E,P) / \ker(\sigma)$, i.e., $\sigma$ is a complete quotient map.
\end{proof}
In view of all the above, we finally propose the following.

\begin{dfn}\label{extension dfn}
Let $F\subseteq\mathcal{B}(H_F)$ and $E\subseteq\mathcal{B}(H_E)$ be unital operator systems. A Toeplitz type extension is the unital operator system
\[
X(F;E,P):= \sk(PH_E)\otimes_{\min} F + PEP\otimes \mathbb{C}\cdot 1_{H_{F}}
\]
in $\mathcal{B}(PH_E\otimes H_F)$, where $P\in\mathcal{B}(H_E)$ is an infinite-dimensional projection satisfying
\begin{enumerate}[$(i)$]
\item $[P,e]\in\sk(H_E)$ for all $e\in E$;
\item If $c\in C^*(E)$ be such that $PcP\in\sk(PH_E)$, then $c=0$.
\end{enumerate}
\end{dfn}
The most fundamental example is $\sk(\ell^2(\bbn))+PC^*(\bbz)_{(n)}P$, where $C^*(\bbz)_{(n)},\,n\ge 2,$ is the Fej\'er-Riesz operator system and $P:L^2(S^1)\cong\ell^2(\bbz)\to\ell^2(\bbn)$ is the Hardy projection.

\begin{dfn}\label{extension process}
By Toeplitz type iterative extension process, we mean the sequence of unital operator systems $\{X^j(F,E,P)\}_{j=0}^\infty$ obtained by the iterative process $$X^0(F;E,P):=X(F;E,P)\,,\quad X^j(F;E,P):=X(X^{j-1}(F;E,P);E,P)\,\mbox{ for }j\ge 1.$$
\end{dfn}


\section{Hyperrigidity of Toeplitz type Extension}\label{Sec5}

\subsection{Representations of Toeplitz type \texorpdfstring{$C^*$}{c*}-extension}\label{Sec5.1}

Recall Toeplitz type $C^*$-extensions from \Dref{Hawkins Zacharias defn}. Let $(A,H_A,P)\in \mathrm{Op}_{C^*},$ i.e. $A\subseteq \mathcal{B}(H_A),$ and $B\subseteq \mathcal B(H_B)$ be a unital $C^*$-algebra. Consider $$Y:=\sk(PH_A)\otimes B+PAP\otimes 1_{H_B}=J+PAP\otimes 1_{H_B}$$ with $J:=\sk(PH_A)\otimes B$. Then $Y$ is a unital $C^*$-subalgebra of $\mathcal B(PH_A\otimes H_B)$ with the unit $P\otimes 1_{H_B}$. Let $\pi:Y\to \mathcal{B}(H)$ be a unital representation of $Y$ and $H_1:= \overline{\pi(J)(H)}$. Then $H=H_1\oplus H_1^\perp=:H_1\oplus H_2$ (either summand may be zero). Because $J\trianglelefteq Y$ is an ideal, it follows that $H_1$ reduces $\pi$, i.e., if $P_{H_1}$ is the orthogonal projection onto $H_1$, then $P_{H_1}\pi(y)=\pi(y)P_{H_1}$ for all $y\in Y.$ Thus, we have $\pi=\pi_1\oplus\pi_2$, where $\pi_1:Y\to\mathcal B(H_1)$ and $\pi_2:Y\to\mathcal B(H_2)$ are the restrictions of $\pi:Y\to\mathcal B(H_1\oplus H_2)$ to the reducing subspaces $H_1$ and $H_2$, respectively.

\begin{lmma}
$\pi_1\big|_J:J\to \mathcal B(H_1)$ is a non-degenerate representation of $J$, and $\pi_2(J)=0.$
\end{lmma}
\begin{proof}
Let $(u_\lambda)$ be an approximate identity for $J$. We claim that $\xi\in H_1$, $\pi_1(u_\lambda)\xi\to\xi$. Indeed, this is clear for $\xi=\pi(j)\eta$, $j\in J$, since $\pi_1(u_\lambda)\pi(j)\eta=\pi(u_\lambda j)\eta\to\pi(j)\eta$, and then it holds for all $\xi\in H_1$ by continuity. Since $\pi_1(u_\lambda)\xi\in\pi_1(J)H_1$, we obtain $\overline{\pi_1(J)H_1}=H_1.$ Thus $\pi_1\big|_J$ is non-degenerate.

For the second part, let $\xi\in H_2$ and $j\in J$. Since $\pi(j)\xi\in H_1$, while $H_2$ reduces $\pi(Y)$, we also have $\pi(j)\xi\in H_2$. Hence $\pi(j)\xi\in H_1\cap H_2=\{0\}.$ Therefore $\pi_2(j)=0$.   
\end{proof} 

\begin{ppsn}\label{hyperlemma1}
Let $\sigma: Y\to A$ be the quotient homomorphism given by $\sigma(k+PaP\otimes 1_{H_B}):= a$. For a unital representation $\pi: Y\to\mathcal B(H)$, there exist unital representations $\pi_A: A\to\mathcal B(K_1)$ and $\pi_B: B\to \mathcal B(K_2)$ such that $\pi$ is unitarily equivalent to $(\pi_A\circ \sigma)\oplus(\mathrm{id}\otimes\pi_B): Y \to \mathcal B(K_1\oplus (PH_A\otimes K_2))$, given by $$\big((\pi_A\circ \sigma)\oplus (\mathrm{id}\otimes\pi_B)\big)\left(k+PaP\otimes1_{H_B}\right)=\begin{pmatrix} 
\ \  \pi_A(a) & 0\\
0 & (\mathrm{id}_{\sk(PH_A)}\otimes\pi_B)(k)+PaP\otimes 1_{K_2}
\end{pmatrix}$$ (we allow $K_1$ and $K_2$ to be zero Hilbert spaces).
\end{ppsn}
\begin{proof}
We know that $\pi=\pi_1\oplus \pi_2$, where $\pi_1:Y\to \mathcal{B}(H_1)$ and $\pi_2:Y\to \mathcal{B}(H_2)$ are the restrictions of $\pi:Y\to\mathcal B(H_1\oplus H_2)$ . Since \(\pi_2(J)=0\), there exists a representation \(\pi_A: A\to \mathcal B(K_1)\), with $K_1:=H_2$, such that $\pi_2= \pi_A\circ \sigma$, where \(\sigma: Y\to A\) is the quotient homomorphism with kernel $J$.

Regarding \(\pi_1\), since \(\pi_1\big|_J\) is a non-degenerate representation of \(J=\sk(PH_A)\otimes B\), there exist a unitary \(V:H_1\to PH_A\otimes K_2\) and a representation \(\pi_B:B\to \mathcal B(K_2)\) such that \((\mathrm{id}_{\sk(PH_A)}\otimes \pi_B) (k)=V\pi_1\big|_{J}(k)V^*\) for all \(k\in \sk(PH_A)\otimes B\). We now compute \(V\pi_1(PaP\otimes 1_{H_B})V^*\) for \(a\in A\). Observe that for all \(k\in \sk(PH_A)\), we have
\begin{IEEEeqnarray}{lCl}\label{pp}
V\pi_1(PaP\otimes 1_{H_B})V^*(k\otimes 1_{K_2}) &=& V\pi_1(PaP\otimes 1_{H_B})V^* (\mathrm{id}_{\sk(PH_A)}\otimes \pi_B) (k\otimes 1_{H_B})\nonumber\\
&=& V\pi_1(PaP\otimes 1_{H_B})V^* V\pi_1(k\otimes 1_{H_B})V^*\nonumber\\
&=& V\pi_1(PaPk\otimes 1_{H_B})V^*\nonumber\\
&=& PaPk\otimes 1_{K_2}\nonumber\\
&=& (PaP\otimes 1_{K_2})(k\otimes 1_{K_2}).
\end{IEEEeqnarray}
Let \(\{\xi_n\}_{n\in \mathbb{N}}\) be an orthonormal basis of \(PH_A\), and let \(f_N\in \sk(PH_A)\) be the orthogonal projection onto \(\mathrm{span}\{\xi_n\}_{n=1}^N\). Then we have $f_N\otimes 1_{K_2}\longrightarrow 1_{PH_A\otimes K_2}$ strongly. By \Eref{pp}, we get $$(V\pi_1(PaP\otimes 1_{H_B})V^*-PaP\otimes 1_{K_2})(f_N\otimes 1_{K_2})=0 \quad \text{for all } N\in \mathbb{N},$$ and it follows that $V\pi_1(PaP\otimes 1_{H_B})V^*=PaP\otimes 1_{K_2}.$ Therefore, $\pi_1$ is unitarily equivalent to the representation $\mathrm{id}\otimes\pi_B: Y\to \mathcal B(PH_A\otimes K_2)$ given by$$(\mathrm{id}\otimes\pi_B)(k+PaP\otimes 1_{H_B}):= (\mathrm{id}_{\sk(PH_A)}\otimes \pi_B)(k)+PaP\otimes 1_{K_2},$$ which completes the proof.
\end{proof}

Below we characterize unital irreducible representations of $Y$. Let $\mathrm{Irred}(D)$ denotes the set of irreducible unital  representations of a unital $C^*$-algebra $D$, up to unitary equivalence.

\begin{ppsn}\label{hyperrigidity irreducible repn}
Up to unitary equivalence, the irreducible unital representations of $Y$ are
$$\mathrm{Irred}(Y)=\big\{\pi_A\circ\sigma: \pi_A\in \mathrm{Irred}(A)\big\}\bigsqcup\,\big\{\mathrm{id}\otimes\pi_B: \pi_B\in \mathrm{Irred}(B)\big\}.$$
\end{ppsn}
\begin{proof}
From \Pref{hyperlemma1}, we know that a unital representation $\pi$ of $Y$ is a direct sum of representations $\pi_A\circ\sigma$ and $\mathrm{id}\otimes\pi_B$, so in order for $\pi$ to be irreducible, exactly one of $\pi_A\circ\sigma$ and $\mathrm{id}\otimes\pi_B$ is nonzero.

\textbf{Case 1:} If $\mathrm{id}\otimes\pi_B=0$, then $\pi=\pi_A\circ\sigma$. Since $\sigma:Y\to A$ is a quotient homomorphism, so $\sigma(Y)=A$, and hence $\pi_A\circ\sigma$ is irreducible if and only if $\pi_A$ is irreducible.

\textbf{Case 2:} If $\pi_A\circ\sigma=0$, then $\pi=\mathrm{id}\otimes\pi_B$. We claim that $\mathrm{id}\otimes\pi_B$ is irreducible if and only if $\pi_B$ is irreducible. Suppose $\pi_B$ is not irreducible. Then there exists a nontrivial orthogonal projection $Q\in\mathcal{B}(K_2)$ such that $Q\pi_B(b)=\pi_B(b)Q$ for all $b\in B$. Hence $1_{PH_A}\otimes Q$ is a nontrivial orthogonal projection commuting with $(\mathrm{id}\otimes\pi_B)(y)$ for all $y\in Y$. So $\mathrm{id}\otimes\pi_B$ is not irreducible.

For the converse, suppose $\pi_B$ is irreducible, and let $Q\in \mathcal{B}(PH_A\otimes K_2)$ be an orthogonal projection such that $Q(\mathrm{id}\otimes\pi_B)(y)=(\mathrm{id}\otimes\pi_B)(y)Q$ for all $y\in Y$. In particular, $Q(k\otimes 1_{K_2})=(k\otimes 1_{K_2})Q$ for all $k\in\sk(PH_A)$, and hence $Q=1_{PH_A}\otimes S$, where $S\in\cB(K_2)$ is an orthogonal projection. This is because for any arbitrary Hilbert space $K$, $(\cK(PH_A)\otimes 1_K)'\cap\mathcal{B}(PH_A\otimes K)=\bbc.\,1_{PH_A}\otimes\cB(K)$. Fix any nonzero $T\in\sk(PH_A)$. Then $$(1_{PH_A}\otimes S)(T\otimes \pi_B(b))=(T\otimes \pi_B(b))(1_{PH_A}\otimes S)$$ for all $b\in B$, and hence $T\otimes(S\pi_B(b)-\pi_B(b)S)=0$ for all $b\in B$. Since $T\neq0$, it follows that $S\pi_B(b)=\pi_B(b)S$ for all $b\in B$. As $\pi_B$ is irreducible, $S=0$ or $1_{K_2}$. Thus $Q=0$ or $1_{PH_A\otimes K_2}$, and consequently $\mathrm{id}\otimes\pi_B$ is irreducible.
\end{proof}

\subsection{The Unique Extension Property}\label{Sec5.2}

Recall the unital operator system $X:=X(F;E,P)\subseteq \cB(PH_E\otimes H_F)$ from \Dref{extension dfn}. Also, recall from \Tref{mainresult 0} that
\[
C^*(X)=\sk(PH_E)\otimes C^*(F)+ PC^*(E)P\otimes 1_{H_F}.
\] 
We also know that (\Yref{mainresult 01}) $C^*(X)$ is a Toeplitz type $C^*$-extension of $C^*(E)$ by the stable ideal $\sk(PH_E)\otimes C^*(F)$. 

\begin{dfn}
A unital representation $\pi:C^*(S)\to \cB(H)$ is said to have the unique extension property (UEP) with respect to the unital operator system $S\subseteq B(H_S)$ if given a UCP $\Phi:C^*(S)\to \cB(H)$ such that $\pi\big|_{S}=\Phi\big|_{S}$, we have $\pi=\Phi$. 
\end{dfn}

We recall the following result:

\begin{ppsn}[\text{\cite[Proposition 4.4]{A2}}]
Let \(S\subseteq \cb(H_S)\) be an operator system, and let \(\pi_i:C^*(S)\to\mathcal{B}(H_i)\), \(i=1,2\), be representations such that each $\pi_i$ has the UEP with respect to $S$. Then the direct sum $\pi_1\oplus\,\pi_2:  C^*(S)\to\mathcal{B}(H_1\oplus H_2)$ has the UEP with respect to $S$.
\end{ppsn}

Note that the converse is also true. Suppose $\pi_i:C^*(S)\to\cB(H_i)$, $i=1,2$, are representations such that $\pi_1\oplus\pi_2:C^*(S)\to\cB(H_1\oplus H_2)$ has the UEP with respect to $S$. To see that $\pi_1$ has the UEP with respect to $S$, let $\psi:C^*(S)\to\cB(H_1)$ be a UCP map such that $\psi\big|_S=\pi_1\big|_S$. Then $\psi\oplus\pi_2:C^*(S)\to\cB(H_1\oplus H_2)$ is UCP and $(\psi\oplus\pi_2)\big|_S=(\pi_1\oplus\pi_2)\big|_S$. Hence, by the UEP of $\pi_1\oplus\pi_2$, we have $\psi\oplus\pi_2=\pi_1\oplus\pi_2,$ which gives $\psi=\pi_1$. Similarly, $\pi_2$ has the UEP with respect to $S$. 

\begin{notation}
In accordance to \Sref{Sec5.1}, we set $A:=C^*(E)$ and $B:=C^*(F)$ throughout the rest of this subsection.
\end{notation}

From \Pref{hyperlemma1}, we know that any unital representation $\pi: C^*(X)\to\mathcal B(H)$ is unitarily equivalent to $(\pi_A\circ\sigma)\oplus(\mathrm{id}\otimes\pi_B)$, where $\sigma: C^*(X)\to C^*(E)$ is the quotient homomorphism given by $\sigma(k+PaP\otimes 1_{H_B}):= a$. So $\pi$ has the UEP with respect to $X$ if and only if $(\pi_A\circ\sigma)\oplus(\mathrm{id}\otimes\pi_B)$ has the UEP with respect to $X$. Further,
\begin{IEEEeqnarray}{lCl}\label{hyperrigidity uep of X}
& & (\pi_A\circ\sigma)\oplus(\mathrm{id}\otimes\pi_B)\ \text{has the UEP with respect to }X\nonumber\\
&\Longleftrightarrow& \left\{ \begin{array}{l} \pi_A\circ\sigma\ \text{has the UEP with respect to }X,\\[2mm] (\mathrm{id}\otimes\pi_B)\ \text{has the UEP with respect to }X.
\end{array}
\right.
\end{IEEEeqnarray}

\begin{ppsn}\label{hyperrigidity uep 1}
We have $\pi_A:A=C^*(E)\to \cb(K_1)$ has the UEP with respect to $E$ if and only if $\pi_A\circ \sigma:C^*(X)\to \cb(K_1)$ has the UEP with respect to $X$.
\end{ppsn}
\begin{proof}
Suppose $\pi_A\circ \sigma$ has the UEP with respect to $X$. Let $\phi:C^*(E)\to \cb(K_1)$ be a UCP map such that $\phi\big|_{E}=\pi_A\big|_{E}$. Since $\sigma(X)=E$, we have $(\phi\circ\sigma)\big|_{X}=(\pi_A\circ\sigma)\big|_{X}.$ Hence, by the UEP of $\pi_A\circ\sigma$, we get
$\phi\circ\sigma=\pi_A\circ\sigma.$ Since $\sigma$ is surjective, it follows that $\phi=\pi_A$.

Conversely, suppose that $\pi_A$ has the UEP with respect to $E$. Let $\phi:C^*(X)\to \cb(K_1)$ be a UCP map such that $\phi\big|_{X}=(\pi_A\circ\sigma)\big|_{X}.$ Let $f_N\in\sk(PH_E)$ be as described in the proof of \Lref{helpful lemma} (or \Pref{hyperlemma1}). Then $f_N\otimes1_{H_F}\in X$, and since $\pi_A\circ\sigma$ vanishes on $\sk(PH_E)\otimes C^*(F)$, we have $$\phi(f_N\otimes1_{H_F})=(\pi_A\circ\sigma)(f_N\otimes1_{H_F})=0.$$ Since $f_N\otimes1_{H_F}$ is an orthogonal projection, we have $$\phi((f_N\otimes1_{H_F})^*(f_N\otimes1_{H_F}))=\phi(f_N\otimes1_{H_F})=0=\phi(f_N\otimes1_{H_F})^*\phi(f_N\otimes1_{H_F}).$$ Thus, by \Tref{Multiplicative domain theorem}, $f_N\otimes1_{H_F}\in \mathrm{MD}(\phi)$ (multiplicative domain). Hence \begin{IEEEeqnarray}{lCl}\label{qq}
\phi\big((f_N\otimes1_{H_F})b\big)=\phi(f_N\otimes1_{H_F})\phi(b)=0
\end{IEEEeqnarray}
for all $b\in C^*(X)$. In particular, for $b\in\sk(PH_E)\otimes C^*(F)$, $\phi((f_N\otimes1_{H_F})b)=0.$ By \eqref{Fact1}, we have $\|(f_N\otimes1_{H_F})b-b\|\longrightarrow0$ for all $b\in\sk(PH_E)\otimes C^*(F)$. Therefore, by the continuity of $\phi$, and \Eref{qq}, we have $\phi({\sk(PH_E)\otimes C^*(F))}=0.$ Since $\sigma:C^*(X)\to C^*(E)$ is a surjective $\ast$-homomorphism with $\ker\sigma=\sk(PH_E)\otimes C^*(F)$, by \cite[Proposition 3.6]{KPTT1} we have a UCP map $\widetilde{\phi}:C^*(E)\to\cb(K_1)$ such that $\phi=\widetilde{\phi}\circ\sigma.$ Now, for $e\in E$, as $PeP\otimes 1_{H_F}\in X$, so  $$\widetilde{\phi}(e) =\widetilde{\phi}\bigl(\sigma(PeP\otimes1_{H_F})\bigr)  =\phi(PeP\otimes1_{H_F}) =\pi_A\bigl(\sigma(PeP\otimes1_{H_F})\bigr) =\pi_A(e).$$ Hence $\widetilde{\phi}\big|_{E}=\pi_A\big|_{E}.$ Since $\pi_A$ has the UEP with respect to $E$, it follows that $\widetilde{\phi}=\pi_A.$ Therefore $\phi=\widetilde{\phi}\circ\sigma=\pi_A\circ\sigma,$  which completes the proof.
\end{proof}

\begin{ppsn}\label{hyperrigidity uep 2}
We have $\pi_B:B=C^*(F)\to \cb(K_2)$ has the UEP with respect to $F$ if and only if $\mathrm{id}\otimes\pi_B:C^*(X)\to \cb(PH_E\otimes K_2)$ has the UEP with respect to $X$. 
\end{ppsn}
\begin{proof}
Suppose $\mathrm{id}\otimes\pi_B$ has the UEP with respect to $X$. Let $\psi:C^*(F)\to \cb(K_2)$ be a UCP map such that $\psi\big|_{F}=\pi_B\big|_{F}$. Since $\psi$ is UCP, $$\mathrm{id}_{\cb(PH_E)}\otimes\psi:\cb(PH_E)\otimes_{\min}C^*(F)\to \cb(PH_E\otimes K_2)$$ is also UCP. Therefore, $\phi:=(\mathrm{id}_{\cb(PH_E)}\otimes\psi)\big|_{C^*(X)}:C^*(X)\to \cb(PH_E\otimes K_2)$ is UCP and satisfies $\phi\big|_{X}=(\mathrm{id}\otimes\pi_B)\big|_{X}.$ By the UEP of $\,\mathrm{id}\otimes\pi_B$, we have $\phi=\mathrm{id}\otimes\pi_B$. Now fix any nonzero $T\in\sk(PH_E)$ and let $c\in C^*(F)$. Then $$T\otimes\psi(c)=\phi(T\otimes c)=(\mathrm{id}\otimes\pi_B)(T\otimes c)=T\otimes\pi_B(c).$$ Since $T\neq0$, it follows that $\psi(c)=\pi_B(c)$ for all $c\in C^*(F)$. Thus $\psi=\pi_B$.

Conversely, suppose $\pi_B$ has the UEP with respect to $F$. Let $\phi: C^*(X)\to \cb(PH_E\otimes K_2)$ be a UCP such that $\phi\big|_{X}=(\mathrm{id}\otimes\pi_B)\big|_{X}$. Let $\xi_0\in PH_E$ be a unit vector and let $V_0:K_2\to PH_E\otimes K_2$ be given by $\eta\mapsto \xi_0\otimes \eta$. Let $e_0:=|\xi_0\rangle\langle\xi_0|\in \sk(PH_E)$. We use the following basic facts related to $V_0$,
$$V_0^*(S\otimes T)V_0=\langle S\xi_0,\xi_0\rangle T,\,\,V_0^*V_0=1_{K_2},\,\,V_0AV_0^*=e_0\otimes A,\,\mbox{ and }\,V_0V_0^*=e_0\otimes 1_{K_2}.$$
Define $\psi:C^*(F)\to \cb(K_2)$ by $\psi(c):=V_0^*\phi(e_0\otimes c)V_0$. We claim that $\psi$ is a UCP. For unitality, $\psi(1_{H_F})=V_0^*\phi(e_0\otimes1_{H_F})V_0.$ Since $\phi\big|_{X}=(\mathrm{id}\otimes\pi_B)\big|_{X}$, we have $$ \psi(1_{H_F})=V_0^*(\mathrm{id}\otimes\pi_B)(e_0\otimes1_{H_F})V_0 =V_0^*(e_0\otimes1_{K_2})V_0 =\langle e_0\xi_0,\xi_0\rangle1_{K_2} =1_{K_2}.$$
For complete positivity, observe that $c\mapsto e_0\otimes c$ is $\ast$-homomorphism, $\phi$ is assumed to be UCP, and conjugation by $V_0$ is also cp. Further, for $f\in F$ we know that $e_0\otimes f\in X$, and hence $\phi(e_0\otimes f)=(\mathrm{id}\otimes\pi_B)(e_0\otimes f)=e_0\otimes\pi_B(f)$. Therefore, $$\psi(f)=V_0^*\phi(e_0\otimes f)V_0=V_0^*(e_0\otimes\pi_B(f))V_0=\langle e_0\xi_0,\xi_0\rangle\pi_B(f) =\pi_B(f).$$ This proves that $\psi:C^*(F)\to\cb(K_2)$ is UCP that agrees with $\pi_B:C^*(F)\to\cb(K_2)$ on $F$. By the UEP of $\pi_B$ with respect to $F$, we get $\psi=\pi_B$.

Now we show that $\phi(e_0\otimes c)=e_0\otimes\pi_B(c)$ for all $c\in C^*(F)$. Observe that $e_0\otimes1_{H_F}\in\mathrm{MD}(\phi)$, so for all $z\in C^*(X)$ we have $\phi((e_0\otimes1_{H_F})z)=(e_0\otimes1_{K_2})\phi(z)$ and $\phi(z(e_0\otimes1_{H_F}))=\phi(z)(e_0\otimes1_{K_2}).$ Let $z=e_0\otimes c$, for $c\in C^*(F)$. Clearly, $$ (e_0\otimes1_{H_F}) (e_0\otimes c) =e_0\otimes c =(e_0\otimes c)(e_0\otimes1_{H_F}).$$ Hence,
\begin{IEEEeqnarray*}{lCl}
\phi(e_0\otimes c) &=& \phi((e_0\otimes1_{H_F}) (e_0\otimes c)(e_0\otimes1_{H_F}))\\
&=& (e_0\otimes1_{K_2})\phi(e_0\otimes c)(e_0\otimes1_{K_2})\\
&=& V_0V_0^*\phi(e_0\otimes c)V_0V_0^*\\
&=& V_0\psi(c)V_0^*\\
&=& V_0\pi_B(c)V_0^*\\
&=& e_0\otimes\pi_B(c).
\end{IEEEeqnarray*}
Fixing $c\in C^*(F)$, we have proved that $$\phi\big((|\xi_0\rangle\langle\xi_0|)\otimes c\big)=(|\xi_0\rangle\langle\xi_0|)\otimes\pi_B(c)$$ for all unit vectors $\xi_0\in PH_E$. Then by linearity we have $$ \phi\big((|\xi\rangle\langle\xi|)\otimes c\big) =(|\xi\rangle\langle\xi|)\otimes\pi_B(c) $$ for all $\xi\in PH_E$. Now using the polarization identity $$|\xi\rangle\langle\xi'|=\frac14\sum_{k=0}^{3}i^k|\xi+i^k\xi'\rangle\langle\xi+i^k\xi'|,$$ by linearity we get $$\phi\big((|\xi\rangle\langle\xi'|)\otimes c\big)=(|\xi\rangle\langle\xi'|)\otimes\pi_B(c)$$ for all $\xi,\xi'\in PH_E$. Therefore, for any finite rank operator $k_0\in\sk(PH_E)$, we have $\phi(k_0\otimes c)=k_0\otimes\pi_B(c)$. Then the continuity of $\phi$ gives $\phi(k\otimes c)=k\otimes\pi_B(c)$ for all $k\in\sk(PH_E)$. Therefore, we get
\begin{IEEEeqnarray}{lCl}\label{hyperrigid eq 0}
\phi\big|_{\sk(PH_E)\otimes C^*(F)}=\mathrm{id}_{\sk(PH_E)}\otimes \pi_B=(\mathrm{id}\otimes\pi_B)\big|_{\sk(PH_E)\otimes C^*(F)}.
\end{IEEEeqnarray}
Let $f_N\in\sk(PH_E)$ be as described in the proof of \Lref{helpful lemma}. Then $PePf_N\otimes1_{H_F}\in\sk(PH_E)\otimes C^*(F)$ for all $e\in C^*(E)$, and so for all $N\in\mathbb{N}$ and $e\in C^*(E)$ we have
\begin{IEEEeqnarray}{lCl}\label{hyperrigid eq 1}
\phi(PePf_N\otimes1_{H_F})=(\mathrm{id}\otimes\pi_B)(PePf_N\otimes1_{H_F})=(PeP\otimes1_{K_2})(f_N\otimes1_{K_2}).
\end{IEEEeqnarray}
Since $f_N\otimes1_{H_F}$ is an orthogonal projection and $\phi(f_N\otimes1_{H_F})=f_N\otimes1_{K_2}$ is also an orthogonal projection, we have $$\phi\big((f_N\otimes1_{H_F})(f_N\otimes1_{H_F})^*\big)=\phi(f_N\otimes1_{H_F})\phi(f_N\otimes1_{H_F})^*.$$ Thus, by \Tref{Multiplicative domain theorem}, $f_N\otimes1_{H_F}\in \mathrm{MD}(\phi)$. Hence $$ \phi(b(f_N\otimes1_{H_F})) = \phi(b)\phi(f_N\otimes1_{H_F}) $$ for all $b\in C^*(X)$. In particular, for $e\in C^*(E)$, we have
\begin{IEEEeqnarray}{lCl}\label{hyperrigid eq 2}
\phi(PePf_N\otimes1_{H_F}) = \phi(PeP\otimes1_{H_F})(f_N\otimes1_{K_2})\,, \quad N\in\mathbb{N}.
\end{IEEEeqnarray}
Now \eqref{hyperrigid eq 1} and \eqref{hyperrigid eq 2} together imply $$  \big(\phi(PeP\otimes1_{H_F})-(PeP\otimes1_{K_2})\big)(f_N\otimes1_{K_2})=0 $$ for all $N\in\mathbb{N}$. Since $f_N\otimes1_{K_2}\longrightarrow1_{PH_E\otimes K_2}$ strongly, we get
\begin{equation}\label{hyperrigid eq 3}
\phi(PeP\otimes1_{H_F})=PeP\otimes1_{K_2}=(\mathrm{id}\otimes\pi_B)(PeP\otimes1_{H_F})
\end{equation}
for all $e\in C^*(E)$. Finally, \eqref{hyperrigid eq 0} and \eqref{hyperrigid eq 3} shows $\phi=\mathrm{id}\otimes\pi_B$, which concludes the proof.
\end{proof}

The main result of this subsection is the following.
\begin{thm}\label{0th main result of hyperrigidity}
We have $(\pi_A\circ\sigma)\oplus (\mathrm{id}\otimes\pi_B): C^*(X)\to \cb\big(K_1\oplus (PH_E\otimes K_2)\big)$ has the UEP with respect to $X$ if and only if $\pi_A:C^*(E)\to \cb(K_1)$ has the UEP with respect to $E$ and $\pi_B:C^*(F)\to \cb(K_2)$ has the UEP with respect to $F$.
\end{thm}
\begin{proof}
Combine Propositions \ref{hyperrigidity uep 1},\,\ref{hyperrigidity uep 2}, together with \eqref{hyperrigidity uep of X}.
\end{proof}

\subsection{Hyperrigidity and the Quantum double suspension}\label{Sec5.3}

\begin{dfn}
A unital operator system $S\subseteq \cB(H_S)$ is called hyperrigid in $C^*(S)$ if all unital representations of $C^*(S)$ have the UEP with respect to $S$.
\end{dfn}

An important consequence of \Tref{0th main result of hyperrigidity} is the following:

\begin{thm}\label{first main result of hyperigidity}
The Toeplitz type extension $X:=X(F;E,P)$ is hyperrigid in $C^*(X)$ if and only if $E$ is hyperrigid in $C^*(E)$ and $F$ is hyperrigid in $C^*(F)$.
\end{thm}
\begin{proof}
Suppose that $X$ is hyperrigid in $C^*(X)$. We apply \Pref{hyperlemma1} together with the fact that UEP is invariant under unitary equivalence. Let $A:=C^*(E)$ and $B:=C^*(F)$, and $\pi_A:C^*(E)\to\cb(K_1)$ be any unital representation. Since $X$ is hyperrigid, $\pi_A\circ\sigma$ has the UEP with respect to $X$. By \Pref{hyperrigidity uep 1}, $\pi_A$ has the UEP with respect to $E$. Similarly, let $\pi_B:C^*(F)\to\mathcal{B}(K_2)$ be any unital representation. The hyperrigidity of $X$ implies that $\mathrm{id}\otimes\pi_B$ has the UEP with respect to $X$. \Pref{hyperrigidity uep 2} implies that $\pi_B$ has the UEP with respect to $F$. Thus, both $E$ and $F$ are hyperrigid in $C^*(E)$ and $C^*(F)$, respectively.

Conversely, suppose that $E$ and $F$ are hyperrigid in $A:=C^*(E)$ and $B:=C^*(F)$, respectively. Let $\pi:C^*(X)\to\mathcal{B}(K)$ be any unital representation. By  \Pref{hyperlemma1}, $\pi\simeq(\pi_A\circ\sigma)\oplus(\mathrm{id}\otimes\pi_B)$ for some unital representations $\pi_A$ and $\pi_B$ of $A$ and $B$, respectively. Since $E$ and $F$ are hyperrigid, $\pi_A$ has the UEP with respect to $E$ and $\pi_B$ the UEP with respect to $F$. \Tref{0th main result of hyperrigidity} implies that $(\pi_A\circ\sigma)\oplus(\mathrm{id}\otimes\pi_B)$ has the UEP with respect to $X$, and consequently $X$ is hyperrigid in $C^*(X)$.
\end{proof}

\begin{rmrk}\label{motivation}
\begin{enumerate}[$(i)$]
\item Let $I\trianglelefteq C^*(S)$ be an ideal and $q:C^*(S)\to C^*(S)/I$ be the quotient map. It is well-known that if $S$ is hyperrigid in $C^*(S)$, then $q(S)$ is hyperrigid in $C^*(S)/I$. This gives a quick proof of the fact that if $X=X(F;E,P)$ is hyperrigid in $C^*(X)$, then $E$ is also hyperrigid in $C^*(E)$.
\item If $E$ and $F$ are hyperrigid (so $C^*_{env}(E)=C^*(E)$), by \Tref{mainresult 0} and \Tref{first main result of hyperigidity}, we obtain $$C^*_{env}(X)=\sk(PH_E)\otimes C^*_{env}(F)+PC^*_{env}(E)P\otimes\mathbb{C}\cdot1_{H_F}.$$ In \Sref{Sec6.1}, we shall determine the $C^*$-envelope of general Toeplitz-type extensions.    
\end{enumerate}
\end{rmrk}

\Tref{first main result of hyperigidity} produces many examples of hyperrigid operator systems using the Toeplitz type iterative extension process (see \Dref{extension process}). We shall see in \Sref{Sec6.2} (in particular, \Yref{P} and discussions after it) that this mechanism can produce infinitely many pairwise non-isomorphic hyperrigid operator systems beginning with one pair of such. We need $C^*$-envelope of general Toeplitz type extensions to conclude this.

Now we introduce a special case of Toeplitz type extensions of operator systems. Recall Toeplitz type $C^*$-extensions from \Dref{Hawkins Zacharias defn}.
\begin{lmma}
Toeplitz type $C^*$-extensions can be iterated keeping the same quotient.
\end{lmma}
\begin{proof}
Consider a Toeplitz type $C^*$-extension $\mathscr{C}\cong\sk(PH_A)\otimes B+PAP\otimes\bbc$. Then $\mathscr{C}$ has a faithful representation in $\mathcal{B}(PH_A\otimes H_B)$. Replacing $B$ by $\mathscr{C}$, we obtain $\mathscr{C}_1=\sk(PH_A)\otimes\mathscr{C}+PAP\otimes\bbc$. All the conditions in \Dref{Hawkins Zacharias defn} are trivially satisfied. Continue this process.
\end{proof}

\begin{dfn}[\cite{HS}]
Let $A$ be a unital $C^*$-algebra. The quantum double suspension of $A$, denoted by $\Sigma^2 A$, is the $C^*$-algebra determined by the essential extension
$$
0
\longrightarrow
A\otimes\mathcal K
\longrightarrow
\Sigma^2 A
\overset{\rho}{\longrightarrow}
C(S^1)
\longrightarrow
0,
$$
whose Busby invariant
$$
\beta\colon C(S^1)
\longrightarrow
M(A\otimes\mathcal K)/(A\otimes\mathcal K)
$$
satisfies $\beta(z)=\mu(1_A\otimes S)$. Here $z\in C(S^1)$ is the canonical unitary generator, $S$ is the unilateral shift on $\ell^2(\mathbb N)$, and
$\mu\colon M(A\otimes\mathcal K)\to M(A\otimes\mathcal K)/(A\otimes\mathcal K)
$ is the canonical quotient map.
\end{dfn}

Concretely, let $A$ be a unital $C^*$-algebra and $\mathcal T$ be the Toeplitz algebra generated by the shift $T:e_k\mapsto e_{k-1},\,e_0\mapsto 0$. Put $p=1-T^*T$. The quantum double suspension of $A$ is the unital $C^*$-subalgebra of $A\otimes\mathcal T$ defined by
\[
\Sigma^2 A=C^*\bigl(A\otimes p,\;1\otimes T\bigr)
\subseteq A\otimes\mathcal T.
\]
For our purposes, we set the convention $\Sigma^0:=\mathrm{Id}$.

\begin{lmma}\label{similarity we require}
Let $A$ be a unital $C^*$-algebra. The quantum double suspension $\Sigma^2A$ is a Toeplitz type $C^*$-extension (\Dref{Hawkins Zacharias defn}) after the flip $A\otimes\mathcal K\cong\mathcal K\otimes A$. This identification is also an isomorphism of extensions, up to the automorphism $z\mapsto\bar{z}$ of $C(S^1)$.
\end{lmma}
\begin{proof}
We have the Toeplitz type $C^*$-extension
\[
\mathscr{C}:=\sk(\ell^2(\bbn))\otimes A+PC(S^1)P\otimes\bbc=\sk(\ell^2(\bbn))\otimes A+\mathcal{T}\otimes\bbc
\]
where $\mathcal{T}$ is the Toeplitz algebra. On the other hand, $\Sigma^2A=A\otimes \sk(\ell^2(\bbn))+\bbc\otimes\mathcal{T}$. The flip $A\otimes\mathcal{T}\to\mathcal{T}\otimes A$ is a $\ast$-algebra isomorphism that maps $\Sigma^2A$ to $\mathscr{C}$. The flip sends the ideal $A\otimes\sk(\ell^2(\bbn))$ to $\sk(\ell^2(\bbn))\otimes A$. In both algebras, the quotient by this ideal is $C(S^1)$. The Hawkins--Zacharias quotient map sends $T\otimes 1$ to $\bar{z}$ and the Hong--Szyma\'nski quotient map sends $1\otimes T$ to $z$. Hence, the two quotient maps differ only by conjugation.
\end{proof}

Given any unital $C^*$-algebra $A$, we have the associated sequence $\{\Sigma^{2\ell}A\}_{\ell\in\bbn}$ of iterated quantum double suspensions of it. Applying quantum double suspension on the circle, we obtain the $C^*$-algebra $C(SU_q(2))\cong C(S_q^3),\,0<q<1,$ of continuous functions on the compact quantum group $SU_q(2)$, and iterating the process yields the $C^*$-algebras of continuous functions on odd-dimensional quantum spheres $C(S_q^{2\ell+1})$. Similarly, applying the quantum double suspension on the two-point space, we obtain the $C^*$-algebra $C(S_{q,c}^2)$ of continuous functions on the generic Podle\'s spheres $S_{q,c}^2,\,c\in(0,\infty],$ and iterating the process, we obtain the $C^*$-algebras of continuous functions on even-dimensional quantum spheres $C(S_q^{2\ell})$ which are quotients of odd-dimensional quantum spheres. We remind the readers that for each $\ell\in\bbn$, as $C^*$-algebras, $C(S_q^{2\ell+1})\cong C(S_{q^\prime}^{2\ell+1})$ and $C(S_q^{2\ell})\cong C(S_{q^\prime}^{2\ell})$ for $0<q,q^\prime<1$. See \cite{HS} for proofs of these facts.

\begin{rmrk}
In the context of this article, only the underlying $C^*$-algebras matter, and the odd- and even-dimensional quantum spheres, in each dimension, are topologically $q$-independent. There are other family of quantum spheres also due to Hong--Szyma\'nski \cite{HS1}, which we do not consider here.
\end{rmrk}

\begin{dfn}
We call a sequence $\{E_m\}_{m\in\bbn}$ of operator systems hyperrigid in the sequence of $C^*$-algebras $\{C^*(E_m)\}_{m\in\bbn}$ if each $E_m$ is hyperrigid in $C^*(E_m)$.
\end{dfn}

Recall the Fej\'er-Riesz operator system $C^*(\bbz)_{(n)}$, for any $n\ge 2$, from Example \ref{crucial example}. The $C^*$-algebra it generates is $C(S^1)$.
\begin{lmma}\label{hyperrigidity of Fejer-Riesz}
For any $n\ge 2$, we have $C^*(\bbz)_{(n)}$ is hyperrigid in $C(S^1)$.
\end{lmma}
\begin{proof}
If an operator system contains unitaries generating its $C^*$-algebra, then it is hyperrigid \cite{HK}. Now $C^*(\bbz)_{(n)}$, for any $n\ge 2$, contains the standard unitary generator $z\in C(S^1)$.
\end{proof}

\begin{dfn}\label{qds for operator systems}
Let $F\subseteq\mathcal{B}(H_F)$ be a unital operator system. Consider the sequence $\{C^*(\bbz)_{(n)}\}_{n\ge 2}$ of Fej\'er--Riesz operator systems. The quantum double suspensions of $F$ is the sequence of operator system extensions $\{\Sigma^2F_{(n)}\}_{n\ge 2}$, where $$\Sigma^2F_{(n)}:=\sk(\ell^2(\bbn))\otimes_{\min}F+PC^*(\bbz)_{(n)}P\otimes\bbc.\, 1_{H_F},$$ with $P:L^2(S^1)\cong\ell^2(\bbz)\to\ell^2(\bbn)$ the Hardy projection.
\end{dfn}

For fixed $n\ge 2$, we call $\Sigma^2F_{(n)}$ the quantum double suspension of $F\subseteq\mathcal{B}(H_F)$. This is a special case of Toeplitz type extension of operator systems. By \Tref{mainresult 0}, we have $C^*(\Sigma^2F_{(n)})=\Sigma^2(C^*(F))$ for each $n\ge 2$ (see \Yref{Main theorem 1011} for the $C^*$-envelope version). Moreover, by \Tref{first main result of hyperigidity}, $F$ is hyperrigid in $C^*(F)$ if and only if $\Sigma^2F_{(n)}$ is hyperrigid in $\Sigma^2(C^*(F))$.

\begin{ppsn}\label{QDS hyperrigidity}
Let $F$ be a unital hyperrigid operator system in $A:=C^*(F)$. Then the sequence of unital operator systems $\{\Sigma^{2\ell-2}X_F\}_{\ell=1}^\infty$, where
\[
X_F:=\sk(\ell^2(\bbn))\otimes_{\min}F+PC^*(\bbz)_{(n)}P\otimes\bbc
\]
for fixed $n\ge 2$, is hyperrigid in the sequence $\{\Sigma^{2\ell}A\}_{\ell=1}^\infty$ of quantum double suspensions of $A$.
\end{ppsn}
\begin{proof}
By \Lref{hyperrigidity of Fejer-Riesz} and \Tref{first main result of hyperigidity}, we have $X_F$ is hyperrigid in $C^*(X_F)$. By \Lref{helpful lemma} and \ref{sigma map used in prop number}, \Tref{mainresult 0} together with \Lref{similarity we require} implies that $C^*(X_F)$ is $\Sigma^2A$. This settles the case $\ell=1$. Now, continue the extension process in view of \Tref{first main result of hyperigidity}.
\end{proof}

\begin{crlre}\label{for fruitful remark}
The sequence of operator systems $\{\Sigma^{2\ell-2}(X^{\mathrm{odd}})\}_{\ell=1}^\infty$, where
\[
X^{\mathrm{odd}}:=\sk(\ell^2(\bbn))\otimes_{\min}C^*(\bbz)_{(n)}+PC^*(\bbz)_{(n)}P\otimes\bbc
\]
for each fixed $n\ge 2$, is hyperrigid in the sequence $\{C(S_q^{2\ell+1})\}_{\ell=1}^\infty$ of odd-dimensional quantum spheres, and the sequence of operator systems $\{\Sigma^{2\ell}(X^{\mathrm{even}})\}_{\ell=0}^\infty$, where
\[
X^{\mathrm{even}}:=\sk(\ell^2(\bbn))\otimes_{\min}\bbc^2+PC^*(\bbz)_{(n)}P\otimes\bbc
\]
for each fixed $n\ge 2$, is hyperrigid in the sequence $\{C(S_q^{2\ell+2})\}_{\ell=0}^\infty$ of even-dimensional quantum spheres.
\end{crlre}
\begin{proof}
Fix any $n\ge 2$ and apply \Pref{QDS hyperrigidity} with $F=C^*(\bbz)_{(n)}$, which is hyperrigid in $C(S^1)$. Then, $X^{\mathrm{odd}}=\sk(\ell^2(\bbn))\otimes_{\min}C^*(\bbz)_{(n)}+PC^*(\bbz)_{(n)}P\otimes\bbc$, where $P$ is the Hardy projection, is hyperrigid in $C^*(X^{\mathrm{odd}})=\sk(\ell^2(\bbn))\otimes C(S^1)+PC(S^1)P\otimes\bbc$, which is isomorphic to $C(SU_q(2))$ by \Pref{QDS hyperrigidity}. This proves that $X^{\mathrm{odd}}$ is hyperrigid in $C(SU_q(2))$, and then continue the extension process in view of \Tref{first main result of hyperigidity}.

Similarly, for the second part, begin with $F=\bbc\oplus\bbc$, the $C^*$-algebra of a two-point space, which is trivially hyperrigid in $\bbc^2$, and continue the extension process.
\end{proof}

It would be interesting to explore the $K$-theoretic perspective \cite{S1,S2} of quantum double suspension of operator systems, or more generally, for Toeplitz type extension of operator systems.

\subsection{Arveson's hyperrigidity conjecture}\label{Sec5.4}

Arveson's hyperrigidity conjecture says that if all irreducible unital representations of $C^*(S)$ have the UEP with respect to $S$, then $S$ is hyperrigid in $C^*(S)$ \cite{A2,CL0,CL}. In this case, we say that $S$ satisfies the hyperrigidity conjecture. Obviously if $S$ is already hyperrigid in $C^*(S)$, then it satisfies the hyperrigidity conjecture.

The failure of the hyperrigidity conjecture is first established in \cite{BD}. Subsequently, several counterexamples and an obstruction are constructed \cite{B,BS,CL,S,T}. The main result of this subsection shows that every counterexample to the hyperrigidity conjecture gives rise to a sequence of counterexamples. Thus, applying our construction to each counterexample yields further family of counterexamples.

\begin{ppsn}\label{second main result of hyperrigidity}
If $E$ and $F$ satisfy the hyperrigidity conjecture, then $X:=X(F;E,P)$ also satisfies the hyperrigidity conjecture.
\end{ppsn}
\begin{proof}
Suppose $E$ and $F$ satisfy the hyperrigidity conjecture. Now assume that all unital irreducible representations of $C^*(X)$ have the UEP with respect to $X$. From \Pref{hyperrigidity irreducible repn}, we get that, up to unitary equivalence, each representation in the following set has the UEP with respect to $X$: $$\big\{\pi_{C^*(E)}\circ\sigma: \pi_{C^*(E)}\in \mathrm{Irred}(C^*(E))\big\} \bigsqcup\,\big\{\mathrm{id}\otimes\pi_{C^*(F)}: \pi_{C^*(F)}\in \mathrm{Irred}(C^*(F))\big\}.$$ By \Pref{hyperrigidity uep 1} and \Pref{hyperrigidity uep 2}, we get that all unital irreducible representations of $C^*(E)$ and $C^*(F)$ have the UEP with respect to $E$ and $F$, respectively. By hypothesis, $E$ and $F$ are hyperrigid in $C^*(E)$ and $C^*(F)$, respectively. By \Tref{first main result of hyperigidity}, we get $X$ is hyperrigid in $C^*(X)$, and so $X$ satisfies the hyperrigidity conjecture.
\end{proof}

The following result provides a mechanism for generating new examples and counterexamples to Arveson's hyperrigidity conjecture from the existing ones.

\begin{ppsn}\label{main-BGP}
Let $(E,H_E,P)\in\mathrm{Op}_{\mathrm{OS}}\mbox{ and }(A,H_A,P)\in \mathrm{Op}_{C^*}$. Let $F\subseteq\cb(H_F)$ be a unital operator system and $B\subseteq \cb(H_B)$ be a unital $C^*$-algebra.
\begin{enumerate}[$(i)$]
\item Suppose $E$ is a hyperrigid operator system in $C^*(E)\subseteq\cb(H_E)$. Then $F$ satisfies the hyperrigidity conjecture if and only if $X(F;E,P)$ satisfies the hyperrigidity conjecture.
\item $F$ satisfies the hyperrigidity conjecture if and only if
$X(F;A,P)$ satisfies the hyperrigidity conjecture.
\item Suppose $F$ is a hyperrigid operator system in $C^*(F)\subseteq\cb(H_F)$. Then $E$ satisfies the hyperrigidity conjecture if and only if $X(F;E,P)$ satisfies the hyperrigidity conjecture.
\item $E$ satisfies the hyperrigidity conjecture if and only if $X(B;E,P)$ satisfies the hyperrigidity conjecture.
\end{enumerate}
\end{ppsn}
\begin{proof}
$(ii)$ is immediate from $(i)$, $(iv)$ is immediate from $(iii)$, and the proofs of $(i)$ and $(iii)$ are similar, so we only prove $(i)$. Suppose $X(F;E,P)$ satisfies the hyperrigidity conjecture, and assume that all irreducible unital representations of $C^*(F)$ have the UEP with respect to $F$. Since $E$ is hyperrigid, every unital representation of $C^*(E)$ has the UEP with respect to $E$, in particular,  every irreducible unital representation of $C^*(E)$ has the UEP with respect to $E$. By Propositions \ref{hyperrigidity irreducible repn}, \ref{hyperrigidity uep 1}, \ref{hyperrigidity uep 2}, every irreducible unital representation of $C^*(X)$ has the UEP with respect to $X(F;E,P)$. Since $X(F;E,P)$ satisfies the hyperrigidity conjecture, it follows that $X(F;E,P)$ is hyperrigid in $C^*(X)$. By \Tref{first main result of hyperigidity}, $F$ is hyperrigid in $C^*(F)$, so $F$ satisfies the hyperrigidity conjecture. The converse direction follows from \Pref{second main result of hyperrigidity}.
\end{proof}

\begin{ppsn}\label{1111}
Every counterexample to the hyperrigidity conjecture gives rise to a sequence of counterexamples.
\end{ppsn}
\begin{proof}
Choose any hyperrigid operator system $E$, including degenerate ones, in $\mathrm{Op}_{\mathrm{OS}}$. Let $F\subseteq\mathcal{B}(H_F)$ be an operator system which gives a counterexample to the hyperrigidity conjecture. By part $(i)$ of \Pref{main-BGP}, $X(F;E,P)$ ($P=\mathrm{id}$ is also allowed) gives a counterexample to the hyperrigidity conjecture. Now, keep $E$ fixed and replace $F$ by $X(F;E,P)$ to generate the next counterexample, again applying part $(i)$ of the same proposition. Iterating, we obtain a sequence of counterexamples from one counterexample $F$.
\end{proof}

A particular instance of the above is the following.
\begin{crlre}\label{11111}
Suppose $F\subseteq\mathcal{B}(H_F)$ is a counterexample to the hyperrigidity conjecture. Let $A$ be a unital simple infinite-dimensional $C^*$-algebra with a faithful representation $\pi:A\to\mathcal{B}(H)$ on a separable Hilbert space $H$. All of the following
\[
X_0:=X(F;A,\mathrm{id}_H)=\sk(H)\otimes F+A\otimes\bbc.\,1_{H_F}\,,\qquad X_j:=X(X_{j-1};A,\mathrm{id}_H),
\]
for all $j\ge 1$, are counterexamples to the hyperrigidity conjecture.
\end{crlre}

Thus, we obtain a large class of concrete counterexamples to the hyperrigidity conjecture. In \Sref{Sec6.2}, we shall determine a $K$-theoretic condition, using the $C^*$-envelope of general Toeplitz type extensions of operator systems, under which \Pref{1111} gives infinitely many pairwise non-isomorphic counterexamples arising from one counterexample.
\smallskip

\noindent\textbf{The commutative setting:} The hyperrigidity conjecture is very recently disproved in the commutative setting as well \cite{BS}.

Let $F$ be a function system in a commutative $C^*$-algebra, and suppose that $(E,H_E,P)\in\mathrm{Op}_{\mathrm{OS}}$ is such that $E$ is hyperrigid in $C^*(E)\subseteq\cb(H_E)$. Recall from part $(i)$ of \Pref{main-BGP} that $F$ satisfies the hyperrigidity conjecture if and only if the extension $X(F;E,P)$ does. By construction, $X(F;E,P)$, as well as all its iterated extensions, generate noncommutative $C^*$-algebras. Thus, the validity of the conjecture for function system $F$ depends crucially on its validity for the family of operator systems arising from arbitrary hyperrigid systems in $\mathrm{Op}_{\mathrm{OS}}$ through the iterated extension processes. This family contains infinitely many operator systems whose generating $C^*$-algebras are noncommutative, and if any one of these operator systems fails to satisfy the conjecture, then the conjecture fails for $F$ as well.

This provides a procedure to generate more counterexamples in the commutative setting. Indeed, by identifying a known counterexample within the family of operator systems arising through the iterated extension process, one can deduce that the corresponding function system $F$ also fails to satisfy the conjecture.


\section{\v{S}ilov Ideal and \texorpdfstring{$C^*$}{C*}-envelope for Toeplitz type Extension}\label{Sec6}

Motivated by part $(ii)$ in Remark \ref{motivation}, as well as \Pref{1111}, in this section we obtain the $C^*$-envelope of a general Toeplitz type extension $X:=X(F;E,P)$ (recall \Dref{extension dfn}). Then we shall revisit the hyperrigidity and the hyperrigidity conjecture from Sections \ref{Sec5.3}, \ref{Sec5.4}.

\subsection{\texorpdfstring{$C^*$}{cstar}-envelope for general Toeplitz type Extension}\label{Sec6.1}

Denote the $C^*$-envelope of $X$ by $(C^*_{env}(X),\lambda)$. Let $(C^*_{env}(F),i_{F})$ be the $C^*$-envelope of $F$. Since $(C^*(F),i)$ is a $C^*$-extension of $F$, by the universal property of $C^*$-envelope, there exists a surjective $*$-homomorphism $\pi_{F}:C^*(F)\to C^*_{env}(F)$ such that $\pi_F\circ i=i_F$, and $\ker(\pi_F)$ is the \v{S}ilov ideal of $C^*(F)$ relative to $F$.

Now consider a general situation. Let $A\subseteq\mathcal{B}(H_A)$ and $B\subseteq\mathcal{B}(H_B)$ be two unital $C^*$-algebras, and $\pi: A\to B$ be a surjective $\ast$-homomorphism. We have the following short exact sequence
\[
0\longrightarrow\ker(\pi)\xrightarrow{\,\,\,i\,\,\,}A\xrightarrow{\,\,\,\pi\,\,\,}B\longrightarrow 0.
\]
Since $\sk(PH_E)$ is nuclear, this induces the following short exact sequence:
\begin{IEEEeqnarray}{lCl}\label{exactness}
\hspace*{-1cm}0\longrightarrow\sk(PH_E)\otimes\ker(\pi)\xrightarrow{\mathrm{id}_{\sk(PH_E)}\otimes\,i}\sk(PH_E)\otimes A\xrightarrow{q_\sk:=\mathrm{id}_{\sk(PH_E)}\otimes\,\pi}\sk(PH_E)\otimes B\longrightarrow 0.
\end{IEEEeqnarray}
Consider the following $\ast$-homomorphism
\[
q:= \mathrm{id}_{\mathcal{B}(PH_E)}\otimes\,\pi: \mathcal{B}(PH_E)\otimes_{\min} A\longrightarrow \mathcal{B}(PH_E)\otimes_{\min} B
\]
and recall from \Lref{helpful lemma}, $\mathcal{C}(X_A):=\sk(PH_E)\otimes A+ PC^*(E)P\otimes 1_{H_A}$, and similarly $\mathcal{C}(X_B)$. Define 
\[
q_\pi: \mathcal{C}(X_A)\longrightarrow \mathcal{C}(X_B),\quad q_\pi(k+PcP\otimes 1_{H_A}):= q_\sk(k)+PcP\otimes1_{H_B}
\]
for $k\in\sk(PH_E)\otimes A$ and $c\in C^*(E)$, where $q_{\sk}$ is as defined in \Eref{exactness}.

\begin{lmma}\label{useful lemma2}
We have $q_\pi=\left.q\right|_{\mathcal{C}(X_A)}$, and $q_\pi$ is a surjective $\ast$-homomorphism with
\[
\ker(q_\pi)=\ker(q_{\sk})=\sk(PH_E)\otimes\ker(\pi).
\]
\end{lmma}
\begin{proof}
The surjectivity of $q_\pi$ follows from that of $q_\sk$. For $k+PcP\otimes 1_{H_A}\in \mathcal{C}(X_A)$, we have
\begin{IEEEeqnarray*}{lCl}
q_\pi(k+PcP\otimes 1_{H_A}) &=& q_\sk(k)+PcP\otimes1_{H_B}\\
&=& q(k)+q(PcP\otimes 1_{H_A})\\
&=& q(k+PcP\otimes 1_{H_A}),
\end{IEEEeqnarray*}
which implies that $q_\pi= \left.q\right|_{\mathcal{C}(X_A)}$, and consequently $q_\pi$ is a $\ast$-homomorphism.

Now, let $-k+PcP\otimes 1_{H_A}\in \ker(q_\pi)$. Then $q_\sk(-k)+PcP\otimes1_{H_B}=0$. Thus, $q_\sk(k)=PcP\otimes1_{H_B}\in (\sk(PH_E)\otimes B) \cap (PC^*(E)P\otimes 1_{H_B})=\{0\}$, where the last equality follows by taking $D=B$ in \Lref{helpful lemma}, part $(ii)$. So $PcP\otimes 1_{H_B}=0$, and by the definition of $\mathrm{Op}_{\mathrm{OS}}$, we obtain $c=0$. This proves $\ker(q_\pi)\subseteq \ker(q_\sk)$. The reverse containment is obvious, and hence $\ker(q_\pi)=\ker(q_\sk)$. By the exactness of \Eref{exactness}, we have $\ker(q_\sk)=\sk(PH_E)\otimes\ker(\pi)$, which completes the proof.
\end{proof}

Putting $A=C^*(F)$, $B=C^*_{env}(F)$, and $\pi=\pi_F:C^*(F)\to C^*_{env}(F)$ in Lemma \ref{useful lemma2}, we obtain the following.
\begin{crlre}\label{q_pi_F}
There is a surjective $\ast$-homomorphism 
\begin{IEEEeqnarray*}{lCl}
q_{\pi_F} : C^*(X)=\sk(PH_E)\otimes C^*(F) + PC^*(E)P\otimes \bbc.1_{H_{F}} &\longrightarrow& \sk(PH_E)\otimes C^*_{env}(F) +\nonumber\\
& & PC^*(E)P\otimes \bbc.1_{H_{C^*_{env}(F)}}
\end{IEEEeqnarray*}
with kernel $\ker({q_{\pi_F}})=\sk(PH_E)\otimes\ker({\pi}_{F})$.
\end{crlre}

\begin{lmma}\label{p1}
$\sk(PH_E)\otimes\ker(\pi_F)$ is a boundary ideal of $C^*(X)$ relative to $X$.
\end{lmma}
\begin{proof}
Let $B:=C^*_{env}(F)$. By \Yref{q_pi_F}, we have a surjective $\ast$-homomorphism $q_\pi:C^*(X)\to \mathcal{C}(X_B)$ with kernel $\ker({q_{\pi}})= \sk(PH_E)\otimes\ker({\pi})$. We need to show that $q_\pi$ is a complete isometry on $X$. Since $i_F: F\to B$ is a unital complete order injection, Theorem \ref{P2} implies that $$\mbox{id}_{\mathcal{B}(PH_E)}\otimes_{\text{min}} i_F: \mathcal{B}(PH_E)\otimes_{\text{min}} F\longrightarrow\mathcal{B}(PH_E)\otimes_{\text{min}} B$$ is a unital complete order injection, and therefore by Lemma \ref{P1} it is a complete isometry. Now, let $\sum_{i=1}^{m} k_i\otimes f_i + PeP\otimes 1_{H_{F}} \in \sk(PH_E)\odot F+ PEP\otimes \bbc.1_{H_F}\subseteq X\subseteq  \mathcal{B}(PH_E)\otimes_{\text{min}} F$. Then
\begin{IEEEeqnarray*}{lCl}
& & \left.(\mathrm{id}_{\mathcal{B}(PH_E)}\otimes_{\min} i_F)\right|_{X}\big(\,\textstyle\sum_{i=1}^{m} k_i\otimes f_i + PeP\otimes 1_{H_{F}}\big)\\
&=& \textstyle\sum_{i=1}^{m} k_i\otimes i_F(f_i) + PeP\otimes 1_{H_{B}}\\
&=& \textstyle\sum_{i=1}^{m} k_i\otimes (\pi\circ i)(f_i) + PeP\otimes 1_{H_{B}}\\
&=& \textstyle\sum_{i=1}^{m} k_i\otimes (\left.\pi\right|_{F})(f_i) + PeP\otimes 1_{H_{B}}\\
&=& \left.(id_{\sk(PH_E)}\otimes \pi)\right|_{\sk(PH_E)\otimes_{\min} F} \big(\,\textstyle\sum_{i=1}^{m} k_i\otimes f_i\big) + PeP\otimes1_{H_{B}}\\
&=& \left.q_{\pi}\right|_X \big(\,\textstyle\sum_{i=1}^{m} k_i\otimes f_i+PeP\otimes 1_{H_F}\big).
\end{IEEEeqnarray*}
By the continuity of $\,\mathrm{id}_{\mathcal{B}(PH_E)}\otimes_{\text{min}} i_F: \mathcal{B}(PH_E)\otimes_{\text{min}} F\to  \mathcal{B}(PH_E)\otimes_{\text{min}} B$ and $q_\pi:C^*(X)\to \mathcal{C}(X_B)$, we have 
$\left(\mbox{id}_{\mathcal{B}(PH_E)}\otimes_{\text{min}} i_F)\right|_{X}=\left.q_{\pi}\right|_X.$
Since the left-hand side is the restriction of a complete isometry, $\left.q_{\pi}\right|_X$ is a complete isometry.
\end{proof}

Our next goal is to show that $\sk(PH_E)\otimes \ker(\pi_F)$ is the \v{S}ilov ideal of $C^*(X)$ relative to $X$. For this, let $J$ be any other boundary ideal of $C^*(X)$ relative to $X$. Set $K:=\sk(PH_E)\otimes C^*(F)$. Then $J\cap K$ is a closed two-sided ideal of $K$, and hence $J\cap K$ is of the form $\sk(PH_E)\otimes I$ for some closed two-sided ideal $I$ of $C^*(F)$. Thus, we have the following short exact sequence:
\[
0
\longrightarrow
I
\xrightarrow{}
C^*(F)
\xrightarrow{q_I}
C^*(F)/I
\longrightarrow
0.
\]
Since $\sk(PH_E)$ is a nuclear $C^*$-algebra, we obtain the short exact sequence
\[
0
\longrightarrow
J\cap K
\xrightarrow{}
K
\xrightarrow{id_{\sk(PH_E)}\ \otimes \ q_I}
\sk(PH_E)\otimes (C^*(F)/I)
\longrightarrow
0.
\]
Therefore, $K/(J\cap K)$ is $\ast$-isomorphic to $\sk(PH_E)\otimes (C^*(F)/I)$. Let $q_{J}: C^*(X)\to C^*(X)/J$ be the quotient map. Then, the following diagram commutes:
\begin{equation}\label{commutative diagram}
\begin{tikzcd}[column sep=2em, row sep=2em]
\sk(PH_E) \otimes C^*(F) = K \  
  \arrow[r, "\left.q_J\right|_K"] 
  \arrow[drr, "id_{\sk\left(PH_E\right)} \ \otimes \ q_I"'] 
& \ \ \  q_J(K) = (K+J)/J 
  \arrow[r, phantom, "\underset{\text{\scriptsize $*$ iso.}}{\cong}"] 
& K / K \cap J 
  \arrow[d, phantom, "\cong" sloped] 
  \arrow[d, phantom, "\text{\scriptsize $\ast$-iso.\ }" left] \\
& & \sk(PH_E) \otimes (C^*(F) / I)
\end{tikzcd}
\end{equation}
We need the following result to establish the \v{S}ilov ideal of $C^*(X)$ relative to $X$.

\begin{lmma}\label{I is a boundary ideal of C*(F)}
The quotient map $q_I: C^*(F)\to C^*(F)/I$ is a complete isometry on $F$. Thus, $I$ is a boundary ideal of $C^*(F)$ relative to $F$, which further implies that $I\subseteq \ker(\pi_F)$.
\end{lmma}
\begin{proof}
It is enough to show that for each $n\in\bbn,\,\mathrm{id_n}\otimes q_I$ is an isometry on $M_n(\mathbb{C})\otimes_{\min} F$, where $\mathrm{id_n}$ denotes the identity map of $M_n(\mathbb{C})$. Fix any $n\in \mathbb{N}$, and observe that one can view $M_n(\mathbb{C})\subseteq\mathcal{K}(PH_E)$ as a subalgebra. This implies $M_n(\mathbb{C})\otimes_{\min} F\subseteq \mathcal{K}(PH_E)\otimes_{\min} F\subseteq X$. Since the diagram \ref{commutative diagram} commutes, by using the facts that $q_J$ is a complete isometry on $X$ and any $\ast$-algebra isomorphism is a complete isometry, we have $\mathrm{id_{\mathcal{K}(PH_E)}}\otimes q_I$ is a complete isometry on $\sk(PH_E)\otimes F$, and hence on $M_n(\mathbb{C})\otimes_{\min} F$.
\end{proof}

\begin{ppsn}\label{Silov ideal claim}
The \v{S}ilov ideal of $C^*(X)$ relative to $X$ is $\sk(PH_E)\otimes \ker(\pi_F)$.
\end{ppsn}
\begin{proof}
Let $J$ be any other boundary ideal of $C^*(X)$ relative to $X$. Denote $K:=\sk(PH_E)\otimes C^*(F)$, which is an (essential) ideal of $C^*(X)$ by \Tref{mainresult 0}. Then, $J\cap K=\sk(PH_E)\otimes I$ is a closed two-sided ideal of $K$ for some closed two-sided $I$ of $C^*(F)$, and \Lref{I is a boundary ideal of C*(F)} implies the following:
\begin{equation}\label{equation1}
J\cap K=\sk(PH_E)\otimes I\subseteq \sk(PH_E)\otimes \ker(\pi_F).
\end{equation}
Consider $M:= J+ \sk(PH_E)\otimes \ker(\pi_F)$, which is a  two-sided ideal of $C^*(X)$. We claim that 
\begin{equation}\label{equation2}
M\cap K= \sk(PH_E)\otimes \ker(\pi_F)=\ker(q_{\pi_F}).
\end{equation}
Since $K=\sk(PH_E)\otimes C^*(F)\supseteq \sk(PH_E)\otimes \ker(\pi_F)$ and $M$ by definition contain $\sk(PH_E)\otimes \ker(\pi_F)$, we have $M\cap K\supseteq \sk(PH_E)\otimes \ker(\pi_F)$. To show ``$\subseteq$'', let $x\in M\cap K$. Then $x=\alpha +\beta$ for some $\alpha\in J$ and $\beta\in \sk(PH_E)\otimes \ker(\pi_F)\subseteq K$, so $\alpha=x-\beta$. Since $x,\beta\in K$, we have $\alpha\in K$, which implies that $\alpha \in J\cap K\subseteq\sk(PH_E)\otimes \ker(\pi_F)$ by \eqref{equation1}. Therefore, $\alpha, \beta\in \sk(PH_E)\otimes \ker(\pi_F)$, and so $x\in \sk(PH_E)\otimes \ker(\pi_F)$. Consequently, we have $M\cap K\subseteq \sk(PH_E)\otimes \ker(\pi_F)$. This completes the proof of the claim in \Eref{equation2}.

Let $B:=C^*_{env}(F)$. Since $M$ and $K$ are ideals in $C^*(X)$, their images $q_{\pi_F}(M)$ and $q_{\pi_F}(K)$ are ideals in $\mathcal{C}(X_B)=\sk(PH_E)\otimes B+PC^*(E)P\otimes 1_{H_{B}}$, where $q_{\pi_F}$ is as in \Yref{q_pi_F}. From equation \eqref{equation2}, we obtain the following:
\begin{equation}\label{equation 3}
q_{\pi_F}(M)\cap q_{\pi_F}(K)=q_{\pi_F}(M\cap K)=q_{\pi_F}(\ker(q_{\pi_F}))=\{0\}.
\end{equation}
Now, the part $(iv)$ of Lemma \ref{helpful lemma} shows that the ideal $q_{\pi_F}(K)=\sk(PH_E)\otimes C^{*}_{env}(F)$ is essential.
Thus, equation \eqref{equation 3} forces $q_{\pi_F}(M)=\{0\}$, and consequently $$M=J+ \sk(PH_E)\otimes \ker(\pi_F)\subseteq \ker(q_{\pi_F})=\sk(PH_E)\otimes \ker(\pi_F)$$ by \Yref{q_pi_F}. This shows that $J\subseteq \sk(PH_E)\otimes \ker(\pi_F)$ and completes the proof.
\end{proof}

\begin{thm}\label{Main theorem 1}
Let $i$ be the inclusion of $X:=X(F;E,P)$ in the $C^*$-algebra generated by it inside $\mathcal{B}(PH_E\otimes H_F)$, which is $C^*(X)=\sk(PH_E)\otimes C^*(F)+ PC^*(E)P\otimes \bbc.1_{H_F}$. Let 
\[
q_{\pi_F}:C^*(X) \longrightarrow \sk(PH_E)\otimes C^*_{env}(F)+ PC^*(E)P\otimes\bbc. 1_{H_{C^*_{env}(F)}}
\]
be the surjective $\ast$-homomorphism in \Yref{q_pi_F}. The $C^*$-envelope $(C^*_{env}(X),\lambda)$ of $X$ is given by 
\[
C^*_{env}(X)\cong\sk(PH_E)\otimes C^*_{env}(F)+ PC^*(E)P\otimes \bbc.1_{H_{C^*_{env}(F)}}\quad\mbox{and} \quad\lambda=q_{\pi_F}\circ i
\]
where the sum is linearly direct.
\end{thm}
\begin{proof}
Combine \Yref{q_pi_F} and \Pref{Silov ideal claim}.
\end{proof}

\begin{crlre}\label{mainresult 001}
$C^*_{env}(X)$ is a Toeplitz type $C^*$-algebra extension of $C^*(E)$ by the stable ideal $\sk(PH_E)\otimes C^*_{env}(F)$, which is an essential ideal.
\end{crlre}
\begin{proof}
Part $(ii)$ of \Lref{helpful lemma}, and the fact that $[P,E]\subseteq\mathcal{K}(H_E)$ implies $[P,C^*(E)]\subseteq\mathcal{K}(H_E)$, ensures that the map $q:C^*_{env}(X)\to C^*(E)$ given by $x+PcP\otimes 1\mapsto c$ is a well-defined unital surjective $\ast$-homomorphism, and \Tref{Main theorem 1} implies that we have a short exact sequence of $C^*$-algebras
\[
0\longrightarrow \sk(PH_E)\otimes C^*_{env}(F)\longrightarrow C^*_{env}(X)\longrightarrow C^*(E)\longrightarrow 0
\]
with the ideal $\sk(PH_E)\otimes C^*_{env}(F)$ an essential ideal, thus satisfying all the requirements of \Dref{Hawkins Zacharias defn}.
\end{proof}

An immediate consequence of \Dref{qds for operator systems} is the following.
\begin{crlre}\label{Main theorem 1011}
For quantum double suspension of a unital operator system $F\subseteq\mathcal{B}(H_F)$, we have $C^*_{env}(\Sigma^2F_{(n)})=\Sigma^2(C^*_{env}(F))$ for each $n\ge 2$.
\end{crlre}

\subsection{Revisiting hyperrigidity and the hyperrigidity conjecture}\label{Sec6.2}

Let $(E,H_E,P)\in\mathrm{Op}_{\mathrm{OS}}$ be fixed. For any unital operator system $F\subseteq\cb(H_F)$, consider the following recursive chain:
\begin{IEEEeqnarray}{lCl}\label{new definition}
F_1:=F\,,\qquad F_{j+1}:=X(F_j;E,P)\quad\mbox{for }\,j\ge 1.
\end{IEEEeqnarray}
In general, both of the following situations may occur:
\begin{enumerate}[$(i)$]
\item For some $k\neq j$, we have $F_k\cong_{\mathrm{coi}}F_j$.
\item For some $k\neq j$, we have $F_k\not\cong_{\mathrm{coi}}F_j$, but $C^*_{env}(F_k)\cong C^*_{env}(F_j).$
\end{enumerate}

In the following, we provide one situation when it is guaranteed that for any $k\neq j$, we have $C^*_{env}(F_k)\ncong C^*_{env}(F_j)$, and consequently $F_k\not\cong_{\mathrm{coi}}F_j$.

\begin{thm}\label{proposition for sequence of distinct elements}
Let $F\subseteq\cb(H_F)$ and $E\subseteq\cb(H_E)$ be unital operator systems such that $C^*(E)\cap\sk(H_E)=\{0\}$, that is, $(E,H_E,\mathrm{id}_{H_E})\in \mathrm{Op}_{\mathrm{OS}}$. Suppose that, for some $i\in\{0,1\}$, we have $$\mathrm{rank}(K_i(C^*_{env}(F)))<\infty\,,\qquad 0<\mathrm{rank}(K_i(C^*(E)))<\infty\,,$$ where $K_0,K_1$ are the $K$-groups. Let $\{F_j\}_{j\in \mathbb{N}}$ be as defined in \Eref{new definition}. Then for any $k\neq j$, we have $C^*_{env}(F_k)\ncong C^*_{env}(F_j),$ and consequently, $F_k\not\cong_{\mathrm{coi}}F_j$.
\end{thm}
\begin{proof}
By \Tref{Main theorem 1}, we have for any $j\geq1$,
\begin{IEEEeqnarray*}{lCl}
C^*_{env}(F_{j+1})=C^*_{env}(X(F_j;E,\mathrm{id}_{H_E})) &\cong& X(C^*_{env}(F_j);C^*(E),\mathrm{id}_{H_E})\\
&=& \sk(H_E)\otimes C^*_{env}(F_j)+C^*(E)\otimes\mathbb C\cdot1_{H_F}.
\end{IEEEeqnarray*}
Thus, for every $j\geq1$, we have the short exact sequence $$ 0\longrightarrow \sk(H_E)\otimes C^*_{env}(F_j) \longrightarrow C^*_{env}(F_{j+1})\overset{\sigma_j}{\longrightarrow} C^*(E) \longrightarrow0, $$ where $\sigma_j:C^*_{env}(F_{j+1})\to C^*(E)$ is the surjective $\ast$-homomorphism given by $\sigma_j(x+e\otimes 1_{H_F})=e$, for $x\in \sk(H_E)\otimes C^*_{env}(F_j)$ and $e\in C^*(E)$.
We have a splitting map $\varphi_j:C^*(E)\to C^*_{env}(F_{j+1})$ given by $\varphi_j(e)=e\otimes1_{H_F}$, which is an injective $\ast$-homomorphism and $\sigma_j\circ\varphi_j=\mathrm{id}_{C^*(E)}$. Hence, the above short exact sequence is split exact. By \cite[Proposition 8.3.6]{Bla}, for $i=0,1$, we have
\begin{IEEEeqnarray*}{lCl}
K_i(C^*_{env}(F_{j+1})) &\cong& K_i\big(\sk(H_E)\otimes C^*_{env}(F_j)\big) \oplus K_i(C^*(E))\\
&\cong& K_i(C^*_{env}(F_j)) \oplus K_i(C^*(E))
\end{IEEEeqnarray*}
since $H_E$ is infinite-dimensional. Thus, for any $k\ge 1$ we have $$K_i(C^*_{env}(F_k)) \cong K_i(C^*_{env}(F)) \bigoplus \big(\oplus_{r=1}^{k-1}\,K_i(C^*(E))\big).$$ Hence $$\mathrm{rank}\big(K_i(C^*_{env}(F_k))\big) = \mathrm{rank}\big(K_i(C^*_{env}(F))\big) +(k-1)\mathrm{rank}\big(K_i(C^*(E))\big).$$ By our hypothesis, if $C^*_{env}(F_k)\cong C^*_{env}(F_j)$, we get $k=j$. Therefore, for $k\neq j$, we have $ C^*_{env}(F_k)\ncong C^*_{env}(F_j)$, and consequently  $F_k\ncong_{\mathrm{coi}}F_j$.
\end{proof}

An immediate consequence is the following.
\begin{crlre}\label{P}
Let $F\subseteq\cb(H_F)$ and $E\subseteq\cb(H_E)$ be unital operator systems as in \Tref{proposition for sequence of distinct elements}. Assume that both are hyperrigid. The iteration process from \Tref{first main result of hyperigidity} produces infinitely many pairwise non-isomorphic hyperrigid operator systems.
\end{crlre}

In particular, consider the Fej\'er--Riesz operator system $C^*(\bbz)_{(n)}$ from Example \ref{crucial example} for any $n\ge 2$. We have $C(S^1)\cap\sk(L^2(S^1))=\{0\}$ and $K_i(C(S^1))\cong\mathbb{Z}$ for both $i\in\{0,1\}$. Hence, the tuple $\big(C^*(\bbz)_{(n)},L^2(S^1),\mathrm{id}_{L^2(S^1)}\big)\in\mathrm{Op}_{\mathrm{OS}}$ provides an example of an operator system $E$ satisfying the hypotheses of \Tref{proposition for sequence of distinct elements}, which is moreover hyperrigid in $C^*(E)=C(S^1)$ (\Lref{hyperrigidity of Fejer-Riesz}). Take $F=C^*(\bbz)_{(n)}$, which is also hyperrigid in $C^*(F)=C(S^1)$. The Toeplitz type extension in this case is $X(F;E,\mathrm{id}):=\sk(\ell^2(\bbz))\otimes C^*(\bbz)_{(n)}+C^*(\bbz)_{(n)}\otimes\bbc$. Thus, the hypothesis in \Tref{proposition for sequence of distinct elements}, as well as in \Yref{P}, is satisfied. Iterating this $X(F;E,\mathrm{id})$ produces infinitely many pairwise non-isomorphic hyperrigid operator systems.
\medskip

\noindent\textbf{\textit{Counterexamples to the hyperrigidity conjecture:}} Suppose $F\subseteq \cb(H_F)$ is a counterexample to the hyperrigidity conjecture. In particular, all irreducible representations of $C^*(F)$ have the UEP with respect to $F$, and consequently we have $C^*_{env}(F)=C^*(F)$. Now, consider $\big(C^*(\bbz)_{(n)},L^2(S^1),\mathrm{id}_{L^2(S^1)}\big)\in\mathrm{Op}_{\mathrm{OS}}$ for some $n\ge 2$. Assume that $\mathrm{rank}(K_i(C^*(F)))<\infty$ for at least one $i\in\{0,1\}$. \Tref{proposition for sequence of distinct elements} shows that \Pref{1111} (see also \Yref{11111}) produces infinitely many pairwise non-isomorphic concrete counterexamples to the hyperrigidity conjecture.

The only question that remains whether there exists a counterexample $F\subseteq \cb(H_F)$ for which we have $\mathrm{rank}(K_i(C^*(F)))<\infty$ for at least one $i\in\{0,1\}$. The answer is yes. Let $S_c,\,0<c<1,$ be a counterexample to the hyperrigidity conjecture constructed by Scherer in \cite{S}. He constructed the split exact sequence
$$
0\longrightarrow \sk(L^2(S^3)\oplus\bbc)
\longrightarrow C^*(S_c)
\longrightarrow C(S^3)
\longrightarrow0.
$$
Since $K_0\big(C(S^3)\big)=K_1\big(C(S^3)\big)=\bbz$ and $K_0\big(\sk(L^2(S^3)\oplus\bbc)\big)=\bbz$ and $K_1\big(\sk(L^2(S^3)\oplus\bbc)\big)=0$, we have $\mathrm{rank}(K_0(C^*(S_c)))=2$ and $\mathrm{rank}(K_1(C^*(S_c)))=1$.

Therefore, if we take $F=S_c$, we obtain infinitely many pairwise non-isomorphic concrete counterexamples to the hyperrigidity conjecture through iteration.


\section{Propagation Number of Toeplitz type Extension}\label{Sec7}

In this final section, we compute the propagation number (\Dref{the reqd dfn}) of the unital operator system $X:=X(F;E,P)$ (recall \Dref{extension dfn}). The propagation number is an invariant of operator systems, introduced by Connes and van Suijlekom in \cite{CS}. This has been computed for concrete examples of $C(S^1)^{(n)},\,C^*(\bbz)_{(n)}$ \cite{CS,F1}, and for the operator systems arising through spectral truncations for tori and noncommutative tori $\mathcal{A}_{\Theta}$ \cite{LS,BGP}. The case of stabilization by compacts \cite{CS} and the case of minimal tensor products \cite{Koot} are also investigated.

We first introduce the following notion of the generalized propagation number.

\begin{dfn}\label{generalized prop number}
Let $A$ be a unital operator system and $(B,h)$ be a $C^*$-extension of $A$. The generalized propagation number of $A$ with respect to $(B,h)$ is defined by 
\[
\operatorname{prop}_{(B,h)}(A):=\min\{n\in \mathbb{N}: (h(A))^{\circ n}=B\},
\]
and if no such $n$ exists, we set $\,\operatorname{prop}_{(B,h)}(A):=\infty$.
\end{dfn}

The following lemma shows that $\operatorname{prop}_{(B,h)}(A)$ remains invariant for equivalent $C^*$-extensions of $A$.

\begin{lmma}
Let $A$ and $A_1$ be a unital operator systems such that there exists a unital complete order isomorphism $\phi: A\to A_1$. Let  $(B,\la)$ and $(B_1,\la_1)$ be $C^*$-extensions of $A$ and $A_1$, respectively. Suppose $(B,\la)\sim (B_1,\la_1)$, then 
\[
\operatorname{prop}_{(B,\lambda)}(A) = \operatorname{prop}_{(B_1,\lambda_1)}(A_1).
\]
\end{lmma}
\begin{proof}
Let $\operatorname{prop}_{(B,\lambda)}(A) = n$. By the definition of equivalence, there exists a unital $\ast$-isomorphism $\hat\phi : B \to B_1$ such that $\hat\phi\circ\lambda =  \lambda_1 \circ \phi$. Thus,
\[
(\lambda_1(A_1))^{\circ n} =(\lambda_1( \phi(A)))^{\circ n}= (\hat\phi(\lambda(A)))^{\circ n} = \hat\phi\left((\lambda(A))^{\circ n}\right) = \hat\phi(B) = B_1.
\]
This implies $\operatorname{prop}_{(B_1,\lambda_1)}(A_1) \le n = \operatorname{prop}_{(B,\lambda)}(A)$. The reverse inequality follows by a symmetric argument.     
\end{proof}

\begin{lmma}\label{lemma of generalized prop number} 
Let $(B,h)$ be any $C^*$-extension of $A$. Then
\[
\operatorname{prop}_{(B,h)}(A)\geq \operatorname{prop}_{\left(C^*_{env}(A),i_A\right)}(A)=\operatorname{prop}(A).
\]
\end{lmma}
\begin{proof}
By the universal property of $C^*$-envelopes, there exists a surjective $\ast$-homomorphism $\pi_A: B\to C^*_{env}(A)$ such that $\pi_A\circ h=i_A$. Let $\operatorname{prop}_{(B,h)}(A)=m$. Then, by definition, we have $(h(A))^{\circ m}=B$, which implies that $\pi_A((h(A))^{\circ m})=\pi_A(B)$. Since $\pi_A$ is a surjective $\ast$-homomorphism, we have
\[
C^*_{env}(A)=\pi_A(B)=\pi_A((h(A))^{\circ m})\subseteq (i_A(A))^{\circ m}\subseteq C^*_{env}(A).
\]
This yields us $(i_A(A))^{\circ m}= C^*_{env}(A)$. Therefore, we have $\operatorname{prop}_{(B,h)}(A)=m\geq \operatorname{prop}(A)$.
\end{proof}

Now we prove an important consequence of \Lref{helpful lemma}. Assume the same set-up as in that lemma. Fix a unit vector $\xi\in PH_E$. Define $V_\xi: H_D\to PH_E\otimes\ H_D$ by $V_\xi\eta:=\xi \otimes \eta$. Then $V_\xi$ is an isometry.

\begin{lmma}\label{A1}
The map $\,\tau : \mathcal{C}(X_D)\longrightarrow\mathcal{B}(H_D)$ given by  $\tau(T):= V^*_\xi T V_\xi$ is linear and contractive with $\tau(\mathcal{C}(X_D))=D$.
\end{lmma}
\begin{proof}
Clearly, $\tau$ is linear and contractive. We have
\begin{IEEEeqnarray*}{lCl}
\langle \tau(S\otimes d)(u), v\rangle &=& \langle V^*_\xi (S\otimes d) V_\xi(u),v\rangle \\
&=& \langle (S\otimes d) V_\xi (u), V_\xi (v)\rangle \\
&=& \langle(S\otimes d)(\xi \otimes u), \xi \otimes v\rangle \\
&=& \langle S\xi\otimes du, \xi\otimes v\rangle \\
&=& \langle S\xi,\xi\rangle \cdot \langle du, v\rangle \\
&=& \langle(\langle S\xi,\xi\rangle d)(u),v\rangle \quad \text{for all } u,v \in H_D.
\end{IEEEeqnarray*}
Therefore, $\tau(T)\in D$ for all $T \in \sk(PH_E)\odot D+ PC^*(E)P\otimes \bbc$. Since $\tau$ is contractive, it follows that $\tau(\mathcal{C}(X_D))\subseteq D$. For reverse inclusion, let $S$ be the orthogonal projection onto $\mathbb{C}.\,\xi\subseteq PH_E$. Then $S\otimes d\in \mathcal{C}(X_D)$ and $\tau(S\otimes d)=d$ for all $d\in D$.
\end{proof}

Recall from Section $4$ that for $F\subseteq\mathcal{B}(H_F)$, we have the $C^*$-envelope $(C^*_{env}(F),i_F)$ and $\lambda(X)= \sk(PH_E)\otimes i_F(F) + PEP\otimes \bbc$. Recall the derivation of \Lref{Xon}, replacing $X$ by $\lambda(X)$ and $F$ by $i_F(F)$, to obtain the following:
\begin{equation}\label{lambda X}
(\lambda(X))^{\circ n}=\overline{\sk(PH_E)\otimes (i_F(F))^{\circ n} + PE^{\circ n}P\otimes\bbc} \quad \text{for all } n\in \mathbb{N}. 
\end{equation}
In view of \Tref{Main theorem 1}, setting $D=C^*_{env}(F)$ in \Lref{sigma map used in prop number} and \Lref{A1}, we have the following facts:
\begin{enumerate}[$(i)$]
\item The map $\widetilde\sigma: C^*_{env}(X)\to C^*(E)$, defined by $\widetilde\sigma(k+PcP\otimes 1)=c$, is a surjective $\ast$-homomorphism with $\ker(\widetilde\sigma)=\sk(PH_E)\otimes C^*_{env}(F)$.
\item The map $\tau: C^*_{env}(X)\to C^*_{env}(F)$ is surjective, linear, and contractive.
\end{enumerate}

\begin{thm}\label{main result2}
For the unital operator system $X=\mathcal{K}(PH_E)\otimes_{\min} F+PEP\otimes\bbc\subseteq\mathcal{B}(PH_E\otimes H_F)$, we have $$\operatorname{prop}(X)=\max\big\{\operatorname{prop}(F)\,,\,\operatorname{prop}_{\left(C^{*}(E),i\right)}(E)\big\}\geq\max\{\operatorname{prop}(F)\,,\,\operatorname{prop}(E)\}\,.$$
\end{thm}
\begin{proof}
The inequality `$\ge$' follows from \Lref{lemma of generalized prop number}. Denote $N:=\max\big\{\operatorname{prop}(F)\,,\,\operatorname{prop}_{\left(C^{*}(E),i\right)}(E)\big\}$. First, we show that if $n < N$ (including the case $N=\infty$), then $(\lambda(X))^{\circ n}\neq C^*_{env}(X)$. On the contrary, assume that $(\lambda(X))^{\circ n}= C^*_{env}(X)$ for some $n < N$.

\noindent\textbf{Case 1.} Suppose $n < \operatorname{prop}_{\left(C^{*}(E),i\right)}(E)$, i.e. $E^{\circ n} \subsetneq C^*(E)$.
    
Since $(\lambda(X))^{\circ n}= C^*_{env}(X)$, applying the surjective $\ast$-homomorphism $\widetilde\sigma: C^*_{env}(X)\to C^*(E)$ we get
\[
C^*(E)=\widetilde\sigma (C^*_{env}(X))=\widetilde\sigma\left((\lambda(X))^{\circ n}\right)\subseteq(\widetilde\sigma(\lambda(X)))^{\circ n}.
\]
Recall that $\lambda(X)= \sk(PH_E)\otimes i_F(F) + PEP\otimes \bbc$. Thus, $\widetilde\sigma (\lambda(X))\subseteq E$, which implies that $(\widetilde\sigma(\lambda(X)))^{\circ n}\subseteq E^{\circ n}\subsetneq C^*(E)$. Hence, we have 
\[
C^*(E)=\widetilde\sigma (C^*_{env}(X))=\widetilde\sigma\left((\lambda(X))^{\circ n}\right)\subseteq(\widetilde\sigma(\lambda(X)))^{\circ n}\subseteq E^{\circ n}\subsetneq C^*(E),
\]
a contradiction.

\noindent\textbf{Case 2.} Suppose $n < \operatorname{prop}(F)$, i.e. $(i_F(F))^{\circ n}\subsetneq C^*_{env}(F)$.
    
From \eqref{lambda X}, we have $(\lambda(X))^{\circ n}=\overline{\sk(PH_E)\otimes (i_F(F))^{\circ n} + PE^{\circ n}P\otimes\bbc}$. Applying the map $\tau$ from \Lref{A1}, observe that 
\[
\tau\left(\sk(PH_E)\otimes (i_F(F))^{\circ n} + PE^{\circ n}P\otimes\bbc\right)\subseteq (i_F(F))^{\circ n},
\]
which follows a similar line to \Lref{A1}. Since $\tau$ is surjective, linear, and contractive, we have 
\begin{align*}
C^*_{env}(F) = \tau(C^*_{env}(X))=\tau \left((\lambda(X))^{\circ n}\right)&=\tau\left(\overline{\sk(PH_E)\otimes (i_F(F))^{\circ n} + PE^{\circ n}P\otimes\bbc}\right)\\
&\subseteq \overline{\tau\left(\sk(PH_E)\otimes (i_F(F))^{\circ n} + PE^{\circ n}P\otimes\bbc\right)}\\
&\subseteq (i_F(F))^{\circ n} \subsetneq C^*_{env}(F),
\end{align*}
a contradiction.

This completes the proof that if $n < N$, then $(\lambda(X))^{\circ n}\neq C^*_{env}(X)$, and consequently $\,\operatorname{prop}(X)\geq N$.

Now, we show that if $N < \infty$ and $n \geq N$, then $(\lambda(X))^{\circ n}= C^*_{env}(X)$, and consequently $\,\operatorname{prop}(X)\leq N$. From \eqref{lambda X}, we have $(\lambda(X))^{\circ n}=\overline{\sk(PH_E)\otimes (i_F(F))^{\circ n} + PE^{\circ n}P\otimes\bbc}$. If $n \geq N$, then $(i_F(F))^{\circ n}=C^*_{env}(F)$ and $E^{\circ n}=C^*(E)$. Thus, $(\lambda(X))^{\circ n}=C^*_{env}(X)$, which implies that $\operatorname{prop}(X)\leq N$.
\end{proof}

It is clear that the hyperrigid systems underlying the quantum spheres obtained in \Yref{for fruitful remark} all have infinite propagation number.

\begin{crlre}
If $E$ is hyperrigid, then $\operatorname{prop}(X)=\max\{\operatorname{prop}(E)\,,\,\operatorname{prop}(F)\}$.
\end{crlre}

The following example shows that an upper bound for $\operatorname{prop}(X)$ in \Tref{main result2} is not possible, in general.
\begin{xmpl}\rm
Let $H_E=L^2([0,1])$ and consider the multiplication operator $(M_x\xi)(x)=x\xi(x)$. Set $E:=\operatorname{span}\{1,M_x\}\subseteq \mathcal{B}(H_E)$ and take $P:=I_{H_E}$. Then $C^*(E)=C([0,1])$ and clearly $\operatorname{prop}_{(C^*(E),\iota)}(E)=\infty$, and consequently $\operatorname{prop}(X(E;E,\mbox{Id}))=\infty$. We show that $\operatorname{prop}(E)=1$.

Consider the unital $\ast$-homomorphism $\widetilde{\psi}:C([0,1])\rightarrow\mathbb \bbc^2$ given by $f\mapsto (f(0),f(1))$. The map $\psi:E\rightarrow\mathbb C^2$ defined by $\psi(\alpha+\beta M_x)=(\alpha,\alpha+\beta)$, for $\alpha,\beta\in\bbc$, is the restriction of $\widetilde{\psi}$ to $E$, and hence a UCP. It is clear that $\psi$ is bijective. We show that $\psi^{-1}$ is also completely positive. Observe that $\psi^{-1}(s,t)=s\,1+(t-s)M_x$.
Fix any $n\in\mathbb N$ and consider $(S,T)\in M_n(\mathbb C^2)_+$. Since $M_n(\mathbb C^2)=M_n(\mathbb C)\oplus M_n(\mathbb C)$,
this is equivalent to $S\geq0,\,T\geq0$. The $n$-th amplification of $\psi^{-1}$ is
\[
\psi^{-1,(n)}(S,T)=S\otimes 1+(T-S)\otimes M_x.
\]
Using the canonical identification $M_n(B(H_E))\cong B(H_E^n)$ and $H_E^n=L^2([0,1])^n\cong L^2([0,1],\mathbb C^n)$, the operator $\psi^{-1,(n)}(S,T)$ acts on $\xi=(\xi_1,\cdots,\xi_n)\in L^2([0,1],\mathbb C^n)$ by
\[
\bigl(\psi^{-1,(n)}(S,T)\xi\bigr)(x)=\bigl(S+x(T-S)\bigr)\xi(x).
\]
Thus, it is the multiplication operator associated with the matrix-valued function
\[
\Phi(x):=(1-x)S+xT,\qquad x\in[0,1].
\]
Being a convex combination of the positive matrices $S$ and $T$, we have $\Phi(x)\ge 0$. For $\xi\in L^2([0,1],\mathbb C^n)$,
\[
\left\langle
\psi^{-1,(n)}(S,T)\xi\,,\,\xi
\right\rangle_{L^2}=\int_0^1
\left\langle
\Phi(x)\xi(x)\,,\,\xi(x)
\right\rangle_{\mathbb C^n}\,dx.
\]
Since $\Phi(x)\geq0$ for every $x\in[0,1]$, and $\xi$ was arbitrary, we get that $\psi^{-1,(n)}(S,T)\geq 0$. This completes the proof that $\psi^{-1}$ is completely positive.

Therefore, $\psi:E\rightarrow\mathbb \bbc^2$ is a unital complete order isomorphism. In particular, $(\bbc^2,\psi)$ is a $C^*$-extension of $E$ and $\psi(E)=\bbc^2$. By Lemma \ref{lemma of generalized prop number}, $1\leq \operatorname{prop}(E)\leq \operatorname{prop}_{(\bbc^2,\psi)}(E)=1$ proving that $\operatorname{prop}(E)=1$.\qed
\end{xmpl}

As an immediate consequence of \Tref{main result2}, we obtain the following results.

\begin{crlre}\label{main result22}
If $(E,H_E,P)\in\mathrm{Op}_{C^*}\subseteq\mathrm{Op}_{\mathrm{OS}}$, then $\operatorname{prop}(X)=\operatorname{prop}(F)$.
\end{crlre}
\begin{crlre}\label{main result3}
Let $A$ be a simple infinite-dimensional unital $C^*$-algebra with faithful representation $\pi:A\to\mathcal{B}(H)$. Then, $\operatorname{prop}\big(\sk(H)\otimes_{\min}F+\pi(A)\otimes\bbc\big)=\operatorname{prop}(F)$.
\end{crlre}

The above two corollaries show that the class of Toeplitz type extensions contains
several subclasses, each with a constant propagation number.


\bigskip

\noindent{\sc Vibhor Bhatt} (\texttt{vibhorbhatt25@iitk.ac.in})\\
{\footnotesize Department of Mathematics,
Indian Institute of Technology, Kanpur 208016, India}
\vspace*{.5cm}

\noindent{\sc Satyajit Guin} (\texttt{satyamath@gmail.com; sguin@iitk.ac.in})\\
{\footnotesize Department of Mathematics,
Indian Institute of Technology, Kanpur 208016, India}
\vspace*{.5cm}

\noindent{\sc Palash Sharma} (\texttt{palashsharma8765434124@gmail.com; ms22001@iisermohali.ac.in})\\
{\footnotesize Department of Mathematical Sciences,
Indian Institute of Science Education and Research, Mohali,\\ Knowledge City, Sector 81, Punjab 140 306, India}

\end{document}